\documentclass[10pt]{article}

\usepackage{mathptmx}
\usepackage[T1]{fontenc}
\usepackage{amsmath, amsthm, amsfonts}
\numberwithin{equation}{section}
\usepackage{amssymb}
\usepackage{graphicx}
\usepackage[left=3cm,right=2.5cm, top=3.5cm, bottom=3.5 cm]{geometry}
\newtheorem{thm}{Theorem}[section]
\newtheorem{cor}[thm]{Corollary}
\newtheorem{question}[thm]{Question}
\newtheorem{lem}[thm]{Lemma}
\newtheorem{prop}[thm]{Proposition}
\newtheorem{defn}[thm]{Definition}
\newtheorem{rem}[thm]{Remark}
\newcommand\density[1]{$(#1)$-density}
\newcommand\fixdensity[1]{$#1$-density}
\def\stardensity{$*$-density}
\def\stardensity{$*$-density}
\def\stardense{$*$-dense}

\newcommand\intervalN[2]{[#1\cdots #2)}
\newcommand\intervalA[2]{[#1\cdots #2]}
\def\R{\mathbb{R}}
\def\Z{\mathbb{Z}}
\def\N{\mathbb{N}}
\def\ab{\hspace*{2pt}}
\def\ColN{{\mathsf{T}}}
\def\Col{{\mathsf{Col}}}
\def\Colfast{{\mathsf{Syr}}}
\def\positivenaturals{\mathbb{Z}^+}
\def\oddnaturals{D^+}
\def\cons{\frac{\sqrt{3}}{2}}

\title{An Approximation of the Collatz Map and a Lower Bound for its Average Total Stopping Time}
\author{Manuel Inselmann
}
\date{}

\begin{document}
	\maketitle
	\begin{abstract}
		\noindent
	Define the map $\mathsf{T}$ on the positive integers by $\mathsf{T}(m)=\frac{m}{2}$ if $m$ is even and by $\mathsf{T}(m)=\frac{3m+1}{2}$ if $m$ is odd. Results of Terras and Everett imply that, given any $\epsilon>0$, almost all $m\in\mathbb{Z}^+$ (in the sense of natural density) fulfill $(\frac{\sqrt{3}}{2})^km^{1-\epsilon}\leq \mathsf{T}^k(m)\leq (\frac{\sqrt{3}}{2})^km^{1+\epsilon}$ simultaneously for all $0\leq k\leq  \alpha\log m$ with $\alpha=(\log 2)^{-1}\approx 1.443$. We extend this result to $\alpha=2(\log\frac{4}{3})^{-1}\approx 6.952$, which is the maximally possible value.  Set $\mathsf{T}_{\min}(m):=\min_{n\in\N}\mathsf{T}^n(m)$. As an immediate consequence, one has $\mathsf{T}_{\min}(m)\leq\mathsf{T}^{\left\lfloor2(\log\frac{4}{3})^{-1}\log m\right\rfloor}(m)\leq m^{\epsilon}$ for almost all $m\in\mathbb{Z}^+$ for any given $\epsilon>0$. Previously, Korec has shown that $\mathsf{T}_{\min}(m)\leq m^\epsilon$ for almost all $m\in\mathbb{Z}^+$ if $\epsilon>\frac{\log3}{\log4}$, and recently Tao proved that $\mathsf{T}_{\min}(m)\leq f(m)$ for almost all $m\in\mathbb{Z}^+$ (in the sense of logarithmic density) for all functions $f$ diverging to $\infty$. 
 Denote by $\tau(m)$ the minimal $n\in\mathbb{N}$ for which $\mathsf{T}^n(m)=1$ if there exists such an $n$ and set $\tau(m)=\infty$ otherwise. As another application, we show that $\liminf_{x\rightarrow\infty}\frac{1}{x\log x}\sum_{m=1}^{\lfloor x\rfloor}\tau(m)\geq 2(\log\frac{4}{3})^{-1}$, partially answering a question of Crandall and Shanks. Under the assumption that the Collatz Conjecture is true in the strong sense that $\tau(m)$ is in $O(\log m)$, we show that $\lim_{x\rightarrow\infty}\frac{1}{x\log x}\sum_{m=1}^{\lfloor x\rfloor}\tau(m)= 2(\log\frac{4}{3})^{-1}$.
	\end{abstract}
	\section{Introduction}
	\subsection{Statement of the Main Result}	Let $\positivenaturals$ denote the set of positive integers. The \textbf{Collatz map} is defined by
	$$\Col:\positivenaturals\rightarrow\positivenaturals;\ab m\mapsto\begin{cases}
		\frac{m}{2}&\text{ if\ab\ab} m \text{\ab\ab is\ab\ab even,}\\
		3m+1&\text{ if\ab\ab} m \text{\ab\ab is\ab\ab odd.}
	\end{cases}$$  For  $m\in\positivenaturals$, define $\Col_{\min}(m)=\min_{n\in\N}\Col^n(m)$. The \textbf{Collatz conjecture} states that $\Col_{\min}(m)$ $=1$ for every $m\in\positivenaturals$, see \cite{lagarias} for an overview. The Collatz Conjecture is notoriously difficult and a proof of it seems well out of reach of current mathematics. 
	While it is implausible that probabilistic arguments will lead to a full proof of the Collatz Conjecture, some progress has been made by applying such arguments, for example in \cite{Terras1976}, \cite{EVERETT197742} and more recently in \cite{Tao}. In this paper, we also follow a probabilistic approach.
	
	  Note that whenever $m\in\positivenaturals$ is odd, then $\Col(m)$ is even. Thus, we define $$\ColN:\positivenaturals\rightarrow\positivenaturals;\ab m\mapsto\begin{cases}
	  	\frac{m}{2}&\text{ if\ab\ab} m \text{\ab\ab is\ab\ab even,}\\
	  	\frac{3m+1}{2}&\text{ if\ab\ab} m \text{\ab\ab is\ab\ab odd.}
	  \end{cases}$$ This map is sometimes called the $3x+1$-function (see \cite{kontorovich2009}). A heuristic approach suggests that when we iterate $\ColN$ for $k$ times, roughly half of the elements of the set $$\left\{\ColN^i(m)\mid 0\leq i<k\right\}$$  will be even and the elements of the other half will be odd. Thus, half of the applications of $\ColN$ will be division by $2$ and the other half will be multiplication with approximately $\frac{3}{2}$. Hence, we expect $\ColN^k(m)$ to be close to $\left(\cons\right)^km$ for most $m\in\positivenaturals$. Results of Terras and Everett in \cite{Terras1976} and \cite{EVERETT197742} imply that the above heuristic is correct as long as $0\leq k\leq \log_2m$. Probabilistic models and empirical data suggest that this heuristic can be extended to larger $k$, see \cite[Section 3.3]{lagarias}. Since $\ColN^k(m)\geq 1$ and $\left(\cons\right)^km\geq 1$ holds as long as $k\leq \frac{\log_2m}{1-\log_2\sqrt{3}} $, it seems reasonable to assume that the approximation $\ColN^k(m)\approx\left(\cons\right)^km$ can be extended from $0\leq k\leq \log_2m$ to $0\leq k\leq \frac{\log_2m}{1-\log_2\sqrt{3}} $.
  
  	Building on the work of Terras \cite{Terras1976} and Everett \cite{EVERETT197742}, the main result of this paper is that the approximation $\ColN^k(m)\approx\left(\cons\right)^km$ is indeed correct for $0\leq k\leq \frac{\log_2m}{1-\log_2\sqrt{3}}$ on a set of natural density $1$ (where a set $A\subseteq\positivenaturals$ is of \textbf{natural density} $1$ if $\liminf_{n\rightarrow\infty}\frac{\#\left\{m\in A\mid m\leq n\right\}}{n}=1$). 
	     The main result is:
		\begin{thm}\label{main}
		Suppose that $\epsilon>0$. Then the set $$\left\{m\in\positivenaturals\mid\forall 0\leq k\leq \frac{\log_2m}{1-\log_2\sqrt{3}}:  \left(\cons\right)^km^{1-\epsilon}\leq \ColN^k(m)\leq \left(\cons\right)^km^{1+\epsilon}\right\}$$ is of natural density $1$.
	\end{thm}

\subsection{An Outline of the Proof of the Main Result}
Before we give an outline of the proof of the main result, we define some frequently used notation and concepts.
	\begin{defn}
\begin{enumerate}
	\item 
	Let $F$ be a finite set. By $\#F$ we denote the number of its elements. If $F$ is non-empty, then $\mu_F$ denotes the \textbf{uniform probability measure} on $F$ defined by $\mu_F(B)=\frac{\#B}{\#F}$ for $B\subseteq F$.
	
\item 		Let $a,b\in\N$. We use $\intervalN{a}{b}$ to denote the set $[a,b)\cap\N$, and similarly, we denote with $\intervalA{a}{b}$  the set $[a,b]\cap\N$.

 \item To each $m\in\positivenaturals$, we associate its  \textbf{parity sequence} $(p(m)_k)_{k\in\N}$, which is defined by $$p(m)_k=\begin{cases}
	0& \text{if}\ab\ab \ColN^k(m)\text{\ab is\ab even},\\
	1& \text{if}\ab \ab \ColN^k(m) \text{\ab is\ab odd}.
\end{cases}$$
\item For $x\in \R$, let $\lfloor x\rfloor$  denote the greatest integer $\leq x$.
\item For $x\in \R$, let $\lceil x\rceil$  denote the least integer $\geq x$.
\item Let $f,g:\positivenaturals\rightarrow[0,\infty)$ be functions. We say that $f(m)$ is in $O(g(m))$ if there exists $C>0$ such that $f(m)\leq Cg(m)$ for all $m\in\positivenaturals$.
\end{enumerate}
\end{defn}
It was independently noticed by Terras (see \cite{Terras1976}) and Everett (see \cite{EVERETT197742}) that the parity sequences are distributed like a fair coin flip for the first $n$ steps in each set of the form $\intervalN{m}{m+2^n}$ for $m\in\positivenaturals,\ab n\in\N$. Hence, for most numbers $i\in\intervalN{m}{m+2^n}$, the value $\sum_{k=0}^{n-1}p(i)_k$ will be close to $\frac{n}{2}$. To make this more precise, let $\delta>0$. Then the subset of $i\in\intervalN{m}{m+2^n}$ for which $|\sum_{k=0}^{n-1}p(i)_k-\frac{n}{2}|\geq \delta n$ has less then $C2^{n(1-D)}$ elements for some constants $C>0$ and $0<D<1$ (see Lemma \ref{Hoeffding}). This is a consequence of Hoeffding's inequality (see \cite{Hoeffding}). Sets $S\subseteq \positivenaturals$ with the property that there exist constants  $C>0$ and $0<D\leq 1$ such that $\mu_{\intervalA{1}{N}}(S\cap\intervalA{1}{N})\geq 1-\frac{C}{N^D}$ for all $N\in\positivenaturals$ will be important in this paper. We will say that such a set is \textbf{\stardense} (see Definition \ref{DefStardense}). Note that any \stardense\ab set $S\subseteq \positivenaturals$ is of natural density $1$ as $\lim_{N\rightarrow\infty}\frac{C}{N^D}=0$ for any $D>0$. There are some properties that \stardense\ab sets share with sets of natural density $1$, e.g., finite intersections of \stardense\ab sets are \stardense\ab again. Another important property that will be crucial for the argument is the following.  If a set $S\subseteq \positivenaturals$ is \stardense, then for sufficiently small $\delta>0$ the set \begin{equation}\label{iteration}
	\left\{m\in\positivenaturals\mid \forall k\leq \lfloor\delta \log_3m\rfloor\forall 0\leq i< 2\cdot3^k:\ab 3^km+i\in S\right\}\end{equation} is \stardense\ab as well (see Lemma \ref{leverage}). Note that sets of natural density $1$ do not have this property.

From the initial uniform distribution of the parity sequences it follows that for every $\epsilon >0$ and $0<\alpha\leq 1$ the set \begin{equation}\label{estimationparity}
	\left\{m\in\positivenaturals\mid\forall\ab \lfloor\alpha\log_2m\rfloor\leq k\leq \lfloor\log_2m\rfloor:(\frac{1}{2}-\epsilon)k<  \sum_{i=0}^{k-1}p(m)_i< (\frac{1}{2}+\epsilon)k\right\}\end{equation} is \stardense\ab (see Lemma \ref{paritystar}).
We can write  $$\ColN^k(m)=\left(\frac{m}{2^k}+r_k(m)\right)\cdot 3^{\sum_{i=0}^{k-1}p(m)_i}$$ with  $r_k(m)=\sum_{i=0}^{k-1}\frac{p(m)_i}{3^{\sum_{j=0}^{i}p(m)_j}2^{k-i}}$ (see Lemma \ref{ColIter}). By (\ref{estimationparity}) we can estimate $\sum_{i=0}^{k-1}p(m)_i$ for large enough values of $k$ on a \stardense\ab set.
The next step is to give an estimate of $r_k(m)$ for a \stardense\ab set. In Lemma \ref{rest}, we show that the set \begin{equation}\label{estimationrkm}\left\{m\in\positivenaturals\mid\forall 0\leq k\leq \log_2m: r_k(m) 3^{\frac{k}{2}}m^{-\epsilon}<1\right\}\end{equation} is \stardense\ab for every $\epsilon>0$. In Lemma \ref{beginning}, we show that from the fact that (\ref{estimationparity}) and (\ref{estimationrkm}) are \stardense\ab and a coarse estimate of $\ColN^k(m)$ for small values of $k$ (see Lemma \ref{bruteforce}) it follows that 
 the set \begin{equation}\label{estimationstart}
\left\{m\in\positivenaturals\mid\forall 0\leq k\leq\log_2m:  \left(\cons\right)^km^{1-\epsilon}\leq \ColN^k(m)\leq \left(\cons\right)^km^{1+\epsilon}\right\} \end{equation} is \stardense\ab
 for every $\epsilon>0$.
This is as far as we can get using the uniform distribution of the parity sequence up to $\log_2m$. We know most $m$ in the interval $\intervalN{1}{N}$ are sent to values in $\intervalN{1}{\lfloor N^{\log_2\sqrt{3}}\rfloor}$ via $m\mapsto \ColN^{\lfloor \log_2m\rfloor}(m)$.  Also most $m$ in $\intervalN{1}{\lfloor N^{\log_2\sqrt{3}}\rfloor}$ are sent to values in $\intervalN{1}{\lfloor N^{(\log_2\sqrt{3})^2}\rfloor}$. But it is not immediately clear how one could make an iterative argument work because the distribution of the push-forward of the uniform measure on $\intervalN{M}{M+2^n}$ under $\ColN^n$ is not uniformly distributed. This is where the property that if $S$ is \stardense, then so is the set (\ref{iteration}), comes into play.
In Lemma \ref{imagesareclose}, we show that the set $$\left\{m\in\positivenaturals\mid \exists\ab 0\leq L\leq 2\eta \lfloor\log_2m\rfloor\exists\ab  0\leq i< 2\cdot 3^L: \ColN^{\lfloor\log_2m\rfloor}(m)=\left\lfloor\frac{3^{\left\lfloor(\frac{1}{2}-\eta)\lfloor\log_2m\rfloor\right\rfloor}}{2^{\lfloor\log_2m\rfloor}}m\right\rfloor3^L+i\right\}$$ is \stardense\ab for every $\eta>0$.
We combine this observation with (\ref{iteration}) for suitable $\eta,\delta>0$ and the fact that for any set $S\subseteq\positivenaturals$ that is \stardense, the set $$\left\{m\in\positivenaturals\mid \left\lfloor\frac{3^{\left\lfloor(\frac{1}{2}-\eta)\lfloor\log_2m\rfloor\right\rfloor}}{2^{\lfloor\log_2m\rfloor}}m\right\rfloor\in S\right\}$$ is \stardense\ab as well (see Lemma \ref{scale}). Thus, on the one hand, we know that for most $m$ we can write $$\ColN^{\lfloor\log_2m\rfloor}(m)=\left\lfloor\frac{3^{\left\lfloor(\frac{1}{2}-\eta)\lfloor\log_2m\rfloor\right\rfloor}}{2^{\lfloor\log_2m\rfloor}}m\right\rfloor3^L+i$$ for some small $L,i\in\N$. On the other hand, we know that for most $m$ we have $$\left\lfloor\frac{3^{\left\lfloor(\frac{1}{2}-\eta)\lfloor\log_2m\rfloor\right\rfloor}}{2^{\lfloor\log_2m\rfloor}}m\right\rfloor3^k+j\in S$$ for \textit{all} small $k,j\in\N$.
As a result, we conclude that if $S\subseteq \positivenaturals$ is \stardense, then the set \begin{equation}\label{denseiteration}
\left\{m\in\positivenaturals\mid \ColN^{\lfloor\log_2m\rfloor}(m)\in S\right\}\end{equation} is \stardense\ab as well (see Lemma \ref{Iterate}).
Now, the idea is to apply (\ref{denseiteration}) to the set (\ref{estimationstart}) and intersect the resulting set with (\ref{estimationstart}) (with suitable and possibly distinct choices of $\epsilon$ in (\ref{estimationstart}) in each case). This way, we can extend the range of $k$ from $0\leq k\leq \log_2m$ to $0\leq k\leq (1+\log_2\sqrt{3})\log_2m$. By iterating this operation, we can extend the range of $k$ to $$0\leq k\leq (\sum_{i=0}^n(\log_2\sqrt{3})^i)\log_2m$$ for any $n\in\N$. From this Theorem \ref{main} follows easily.  

\subsection{Applications of the Main Theorem} First, we note a reformulation of Theorem \ref{main}.
	\begin{thm}\label{reformnatural}
	Suppose that $\epsilon>0$. Then the set $$\left\{m\in\positivenaturals\mid\forall \lambda\in[0,1]:  m^{\lambda-\epsilon}\leq \ColN^{\left\lfloor\frac{1-\lambda}{1-\log_2\sqrt{3}}\log_2m\right\rfloor}(m)\leq m^{\lambda+\epsilon}\right\}$$ is of natural density $1$.
\end{thm}
From Theorem \ref{reformnatural} we immediately obtain the following Corollary:
\begin{cor}\label{upperbound}
	Suppose that $\epsilon>0$. Then the set $$\left\{m\in\positivenaturals\mid \ColN^{\left\lfloor\frac{\log_2m}{1-\log_2\sqrt{3}}\right\rfloor}(m)\leq m^{\epsilon}\right\}$$  is of natural density $1$.
\end{cor}
In an attempt to tackle the Collatz Conjecture partially, one may ask the following:
\begin{question}\label{questionf}
	Suppose that $f:\positivenaturals\rightarrow[1,\infty)$ is a function. Is the set $\left\{m\in\positivenaturals\mid\Col_{\min}(m)\leq f(m)\right\}$ of natural density $1$?
\end{question}
As an example, consider the function $f_{\theta}:\positivenaturals\rightarrow(0,\infty);\ab m\mapsto m^\theta$ for $\theta>0$. Allouche showed that the answer to Question \ref{questionf} is positive for all $f_\theta$ with $\theta>\frac{3}{2}-\log_23$ (see \cite{Allouche78}). Korec extended this result to all $f_\theta$ with $\theta>\log_2{\sqrt{3}}$ (see \cite{Korec1994}). Recently, Tao answered Question \ref{questionf} positively for all functions $f:\positivenaturals\rightarrow[1,\infty)$ such that $\lim_{m\rightarrow\infty}f(m)=\infty$ with natural density replaced by logarithmic density (see \cite{Tao} ), where a set  $A\subseteq\positivenaturals$ is of \textbf{logarithmic density} $1$ if $$\liminf_{n\rightarrow\infty}\frac{\sum_{m\in A,\ab m\leq n}\frac{1}{m}}{\sum_{m=1}^{n}\frac{1}{m}}=1.$$ As there are subsets of $\positivenaturals$ with logarithmic density $1$ and lower natural density $0$ (where a set  $A\subseteq\positivenaturals$ is of \textbf{lower natural density} $0$ if $\liminf_{n\rightarrow\infty} \frac{\#A\cap\intervalA{1}{n}}{n}=0$), it would be desirable to get Tao's result for natural density as well.
Now, Corollary \ref{upperbound} implies that the answer to Question \ref{questionf} is positive for  $f_\theta$ for any $\theta>0$, thus extending Korec's result to all $f_\theta$ with $\log_2{\sqrt{3}}\geq\theta>0$, and improving Tao's result for these functions by replacing logarithmic density with natural density. Corollary \ref{upperbound} can be used to answer Question \ref{questionf} positively for some functions that are in $O(m^\theta)$ for every $\theta>0$ as well, but it seems unlikely that the methods of this paper would suffice to prove the full natural density version of Tao's result. 

As another Corollary of Theorem \ref{main} we obtain:

\begin{thm}\label{oddandeven}
	Let $\epsilon>0$. Then the set  $$\left\{m\in\positivenaturals\mid\forall 0\leq k\leq \frac{\log_2m}{1-\log_2\sqrt{3}}:  -\epsilon\log_2m\leq\sum_{i=0}^{k-1}p(m)_i-\frac{k}{2}\leq \epsilon\log_2m\right\}$$ is of natural density $1$.
\end{thm}

For $m\in\positivenaturals$, define its \textbf{total stopping time} $\tau(m)$ to be the minimal $n\in\N$ such that $\ColN^n(m)=1$ if such an $n$ exists and set $\tau(m)=\infty$ otherwise.  Crandall and Shanks conjectured that $$\frac{1}{x}\sum_{m=1}^{\lfloor x\rfloor}\tau(m)\sim \frac{\log_2 x}{1-\log_2\sqrt{3}}$$ (see \cite[page 13]{Lagariasgeneralizations}). We give a partial affirmative answer.
\begin{thm}\label{averagelowerbound}
	It holds that	$\liminf_{x\rightarrow\infty}\frac{1}{x\log_2x}\sum_{m=1}^{\lfloor x\rfloor}\tau(m)\geq\frac{1}{1-\log_2\sqrt{3}}$.
\end{thm}
Of course, Theorem \ref{averagelowerbound} is only interesting in case the Collatz-Conjecture is true. But even if the Collatz conjecture was true, it would still not be clear that the Conjecture of Crandall and Shanks would hold. Kontorovich and Lagarias conjectured that $\tau(m)$ is in $O(\log_2m)$ (see \cite[Conjecture 4.1]{kontorovich2009}). We show that this stronger version of the Collatz conjecture implies the Conjecture of Crandall and Shanks.
\begin{thm}
	Suppose there exists  $C>0$ such that $\tau(m)\leq C\log_2m$ for all $m\in\positivenaturals$, then $$\lim_{x\rightarrow\infty}\frac{1}{x\log_2x}\sum_{m=1}^{\lfloor x\rfloor}\tau(m)= \frac{1}{1-\log_2\sqrt{3}}.$$
\end{thm}
We also note that the results for $\ColN$ imply a similar result for the original Collatz map $\Col$.
\begin{thm}\label{thirdmaintheoremCol}
	Suppose that $\epsilon>0$. Then the set \begin{equation*}
		\left\{m\in\positivenaturals\mid\forall 0\leq k\leq \frac{3\log_2m}{2-\log_23}:  \left(\sqrt[3]{\frac{3}{4}}\right)^{k}m^{1-\epsilon}\leq \Col^k(m)\leq \left(\sqrt[3]{\frac{3}{4}}\right)^km^{1+\epsilon}\right\}\end{equation*} is of natural density $1$.
\end{thm}
The following acceleration of the Collatz map has been subject of investigation as well (for example in \cite{Crandall78} and \cite{Tao}). Given any odd number $m$, in this version of the Collatz map one only pays attention to the odd numbers in the orbit $m,\ColN(m),\ColN^2(m),\cdots$. To be precise, let $\oddnaturals$ denote the set of odd natural numbers, and for any $m\in\positivenaturals$, let $\nu_2(m)$ denote the maximal natural number $k$ such that $2^{k}$ divides $m$. Then define
the \textbf{Syracuse map} (first appearing to the author's knowledge in \cite{Crandall78}) by
\begin{equation*}
	\Colfast:\oddnaturals\rightarrow \oddnaturals;\ab m\mapsto\frac{3m+1}{2^{\nu_2(3m+1)}}.
\end{equation*}
Heuristically, the probability that $\nu_2(3m+1)$ equals $k$, is $2^{-k}$ for $k\geq 1$. Thus, $\log_2(\frac{\Colfast(m)}{m})\approx \log_23-k$ with probability $2^{-k}$. Therefore, one expects $\log_2(\frac{\Colfast(m)}{m})$ to be $\sum_{k=1}^{\infty}(\log_23-k)2^{-k}=\log_23-2=\log_2\frac{3}{4}$. Thus, heuristically, $\Colfast^k(m)\approx \left(\frac{3}{4}\right)^km$ for $0\leq k\leq \left(\log_2\frac{4}{3}\right)^{-1}\log_2m$. We show that this heuristic is indeed correct. For this theorem, define a set $A\subseteq \oddnaturals$ to be of natural density $1$ in $\oddnaturals$ if  $\liminf_{n\rightarrow\infty}\frac{\#\left\{m\in A\mid m\leq 2n+1\right\}}{n+1}=1$.
\begin{thm}\label{TheoremColfast}
	Suppose that $\epsilon>0$. Then the set
	\begin{equation*}
		\left\{m\in \oddnaturals\mid\forall 0\leq k\leq\left(\log_2\frac{4}{3}\right)^{-1}\log_2m:  \left(\frac{3}{4}\right)^{k}m^{1-\epsilon}\leq \Colfast^k(m)\leq \left(\frac{3}{4}\right)^km^{1+\epsilon}\right\}\end{equation*} is of natural density $1$ in $\oddnaturals$.
\end{thm}
	\section{An Approximation of $\ColN$}\label{stopping}
	We begin with some basic results.
The following lemma appears in a slightly different formulation in \cite[Proposition 5.1]{lagarias2006benfords}. It follows from an easy induction, which we spell out for the reader's convenience.
	\begin{lem}\label{ColIter}
For $m\in\positivenaturals$ and $k\in\N$ we have $$\ColN^k(m)=\left(\frac{m}{2^k}+\sum_{i=0}^{k-1}\frac{p(m)_i}{3^{\sum_{j=0}^{i}p(m)_j}2^{k-i}}\right)\cdot 3^{\sum_{i=0}^{k-1}p(m)_i}.$$
\end{lem}
\begin{proof}
	By induction. The case $k=0$ is trivial. If $\ColN^k(m)=\left(\frac{m}{2^k}+\sum_{i=0}^{k-1}\frac{p(m)_i}{3^{\sum_{j=0}^{i}p(m)_j}2^{k-i}}\right)\cdot 3^{\sum_{i=0}^{k-1}p(m)_i}$ and $p(m)_k=0$, then \begin{align*}
	&	\ColN^{k+1}(m)=\frac{1}{2}\left(\frac{m}{2^k}+\sum_{i=0}^{k-1}\frac{p(m)_i}{3^{\sum_{j=0}^{i}P(m)_j}2^{k-i}}\right)\cdot 3^{\sum_{i=0}^{k-1}p(m)_i}\\&=\left(\frac{m}{2^{k+1}}+\sum_{i=0}^{k}\frac{p(m)_i}{3^{\sum_{j=0}^{i}p(m)_j}2^{k+1-i}}\right)\cdot 3^{\sum_{i=0}^{k}p(m)_i},\end{align*} and if $p(m)_k=1$, then \begin{align*}
	\ColN^{k+1}(m)&=\frac{1}{2}\left(3\left(\frac{m}{2^k}+\sum_{i=0}^{k-1}\frac{p(m)_i}{3^{\sum_{j=0}^{i}p(m)_j}2^{k-i}}\right)\cdot 3^{\sum_{i=0}^{k-1}p(m)_i}+1\right)\\
	&=\left(\frac{m}{2^{k+1}}+\sum_{i=0}^{k}\frac{p(m)_i}{3^{\sum_{j=0}^{i}p(m)_j}2^{k+1-i}}\right)\cdot 3^{\sum_{i=0}^{k}p(m)_i}.\end{align*}
\end{proof}
	\begin{defn}
	For $m\in\positivenaturals$ and $k\in\N$, we define $r_k(m)=\sum_{i=0}^{k-1}\frac{p(m)_i}{3^{\sum_{j=0}^{i}p(m)_j}2^{k-i}}$. Thus, we have $$\ColN^k(m)=\left(\frac{m}{2^k}+r_k(m)\right)\cdot 3^{\sum_{i=0}^{k-1}p(m)_i}.$$
	
\end{defn}
\begin{lem}\label{splitrkm}
	Suppose that $m\in\positivenaturals$ and $k,k_0\in\N$. Then
	\begin{enumerate}
		\item $r_{k+k_0}(m)
		=2^{-k}r_{k_0}(m)+\frac{1}{3^{\sum_{j=0}^{k_0-1}p(m)_j}}r_{k}(\ColN^{k_0}(m))$,
		\item  $0\leq r_k(m)< 1$.
	\end{enumerate} 
\end{lem}
\begin{proof}
To see 1., note that
 \begin{align*}
		&	r_{k+k_0}(m)=\sum_{i=0}^{k+k_0-1}\frac{p(m)_i}{3^{\sum_{j=0}^{i}p(m)_j}2^{k+k_0-i}}
		=
		\sum_{i=0}^{k_0-1}\frac{p(m)_i}{3^{\sum_{j=0}^{i}p(m)_j}2^{k+k_0-i}}+\sum_{i=k_0}^{k+k_0-1}\frac{p(m)_i}{3^{\sum_{j=0}^{i}p(m)_j}2^{k+k_0-i}}
		\\&=
		2^{-k}r_{k_0}(m)+\frac{1}{3^{\sum_{j=0}^{k_0-1}p(m)_j}}\sum_{i=k_0}^{k+k_0-1}\frac{p(m)_i}{3^{\sum_{j=k_0}^{i}p(m)_j}2^{k+k_0-i}}
		\\&=	2^{-k}r_{k_0}(m)+\frac{1}{3^{\sum_{j=0}^{k_0-1}p(m)_j}}\sum_{i=0}^{k-1}\frac{p(m)_{i+k_0}}{3^{\sum_{j=0}^{i}p(m)_{j+k_0}}2^{k-i}}\\&=2^{-k}r_{k_0}(m)+\frac{1}{3^{\sum_{j=0}^{k_0-1}p(m)_j}}\sum_{i=0}^{k-1}\frac{p(\ColN^{k_0}(m))_{i}}{3^{\sum_{j=0}^{i}p(\ColN^{k_0}(m))_{j}}2^{k-i}}=2^{-k}r_{k_0}(m)+\frac{1}{3^{\sum_{j=0}^{k_0-1}p(m)_j}}r_{k}(\ColN^{k_0}(m)),\end{align*}
	where we used the fact that $p(m)_{k_0+l}=p(\ColN^{k_0}(m))_l$ for all $l\in\N$.
	
	To see 2., note that
	 $0\leq p(m)_i\leq 1$, $3^{\sum_{j=0}^{i}p(m)_j}\geq 1$, and \ab $\sum_{i=0}^{k-1}\frac{1}{2^{k-i}}<1$. Thus,
	$$0\leq r_k(m)=\sum_{i=0}^{k-1}\frac{p(m)_i}{3^{\sum_{j=0}^{i}p(m)_j}2^{k-i}}\leq\sum_{i=0}^{k-1}\frac{1}{2^{k-i}}< 1.$$
\end{proof}

The following proposition appeared independently in slightly different but equivalent formulations in \cite[Theorem 1.2]{Terras1976} and \cite[Theorem 1]{EVERETT197742}. In the following, $\left\{0,1\right\}^{\intervalN{0}{N}}$ denotes all sequences $(a_i)_{0\leq i<N}$ with values in $\left\{0,1\right\}$.
	\begin{prop}\label{uniformC}
		Suppose that $M\in\positivenaturals$ and $N\in\N$. Then the push-forward measure of the uniform measure on  $\intervalN{M}{M+2^N}$ under the map $m\mapsto (p(m)_i)_{0\leq i<N}$ is the uniform measure on $\left\{0,1\right\}^{\intervalN{0}{N}}$.
	\end{prop}
For a proof, see \cite[Theorem 1.2]{Terras1976} or \cite[Theorem 1]{EVERETT197742}.
\begin{lem}\label{Hoeffding}
	Suppose that $a,b\in\positivenaturals$, $a<b$, $\epsilon>0$, and $0\leq N\leq \lceil\log_2(b-a)\rceil$. Then
	$$\mu_{\intervalN{a}{b}}\left(\left\{m\in\intervalN{a}{b}\mid \left|\sum_{k=0}^{N-1}p(m)_k-\frac{N}{2}\right|\geq\epsilon N\right\}\right)\leq 4e^{-2\epsilon^2N}.$$
\end{lem}

	\begin{proof}
We will use Proposition \ref{uniformC} and Hoeffding's inequality (see \cite{Hoeffding}), which states that if $\nu_n$ is the uniform measure on $\left\{0,1\right\}^{\intervalN{0}{n}}$, then \begin{equation}\label{Hoeffding1}
\nu_{n}\left(\left\{x\in\left\{0,1\right\}^{\intervalN{0}{n}}\mid\left|\sum_{k=0}^{n-1}x_k-\frac{n}{2}\right|\geq\epsilon n\right\}\right)\leq 2e^{-2\epsilon^2n}.\end{equation}
	Set $M=\lceil\log_2(b-a)\rceil$. By Proposition \ref{uniformC},  the push-forward measure of $\mu_{\intervalN{a}{a+2^M}}$ under the map $$\intervalN{a}{a+2^M}\rightarrow\left\{0,1\right\}^{\intervalN{0}{M}};\ab m\mapsto (p(m)_n)_{0\leq n<M}$$ is the uniform measure on $\left\{0,1\right\}^{\intervalN{0}{M}}$, and therefore, also the push-forward of the measure $\mu_{\intervalN{a}{a+2^M}}$ under the map $$\intervalN{a}{a+2^M}\rightarrow\left\{0,1\right\}^{\intervalN{0}{N}};\ab m\mapsto (p(m)_n)_{0\leq n<N}$$  is the uniform measure on $\left\{0,1\right\}^{\intervalN{0}{N}}$. Hence, by Hoeffding's inequality (\ref{Hoeffding1})	$$\mu_{\intervalN{a}{a+2^M}}\left(\left\{m\in\intervalN{a}{a+2^M}\mid \left|\sum_{k=0}^{N-1}p(m)_k-\frac{N}{2}\right|\geq \epsilon N\right\}\right)\leq 2e^{-2\epsilon^2N}.$$ Or equivalently, $$\#\left\{m\in\intervalN{a}{a+2^M}\mid \left|\sum_{k=0}^{N-1}p(m)_k-\frac{N}{2}\right|\geq \epsilon N\right\}\leq 2^{M+1}e^{-2\epsilon^2N},$$ therefore, also $$\#\left\{m\in\intervalN{a}{b}\mid\left| \sum_{k=0}^{N-1}p(m)_k-\frac{N}{2}\right|\geq\epsilon N\right\}\leq 2^{M+1}e^{-2\epsilon^2N}$$ as $\intervalN{a}{b}\subseteq\intervalN{a}{a+2^M}$. Thus, $$\mu_{\intervalN{a}{b}}\left(\left\{m\in\intervalN{a}{b}\mid\left| \sum_{k=0}^{N-1}p(m)_k-\frac{N}{2}\right|\geq\epsilon N\right\}\right)\leq\frac{2^{M+1}}{b-a}e^{-2\epsilon^2N}\leq 4e^{-2\epsilon^2N}$$ as $b-a>2^{M-1}$.
\end{proof}

	In the following, we will frequently deal with subsets $S\subseteq \positivenaturals$ that fulfill a specific notion of density, which we therefore give a name.
	\begin{defn}\label{DefStardense}
		 Suppose that $C>0$, $0<D\leq 1$, and $S\subseteq \positivenaturals$. Then $S$ has \textbf{\density{C,D}} if $$\mu_{\intervalA{1}{N}}(S\cap\intervalA{1}{N})\geq 1-\frac{C}{N^D}$$ for every $N\in\positivenaturals$.  We say that $S\subseteq\positivenaturals$ has \textbf{\fixdensity{D}} if there exists $C>0$ such that $S$ has \density{C,D}, and we say that $S\subseteq\positivenaturals$ is \textbf{\stardense} if $S$ has \fixdensity{D} for some $0<D\leq 1$.
	\end{defn}
We gather some properties of this notion in the following lemma.
\begin{lem}
	Suppose that $S_i\subseteq\positivenaturals$ have \density{C_i,D_i} for some $C_i>0$ and $0<D_i\leq 1$ for $i\in\left\{0,1\right\}$. 
	\begin{enumerate}
		\item The set $S_0\cap S_1$ has \density{C_0+C_1,\min\left(D_0,D_1\right)}.
		\item $S_0$ has \density{C,D} for every $C\geq C_0$ and $0<D\leq D_0$.
		\item Any subset of \ab$\positivenaturals$ containing $S_0$ has \density{C_0,D_0}.
			\item If $S,T\subseteq \positivenaturals$ are \stardense, then $S\cap T$ is \stardense.
	\item For $M\in\positivenaturals$ the set $\positivenaturals\setminus\intervalN{1}{M}$ is \stardense.
	\end{enumerate}
\begin{proof}
	To see $1.$, note that $\frac{C_0}{N^{D_0}}+\frac{C_1}{N^{D_1}}\leq \frac{C_0+C_1}{N^{\min\left(D_0,D_1\right)}}$.
	To see $2.$, note that $\frac{C_0}{N^{D_0}}\leq \frac{C}{N^D}$ in case $C\geq C_0$ and $0<D\leq D_0$.
	$3.$ follows from the fact that $\mu_{\intervalA{1}{N}}(A\cap\intervalA{1}{N})\leq \mu_{\intervalA{1}{N}}(B\cap \intervalA{1}{N})$ whenever $A\subseteq B\subseteq \positivenaturals$.
	$4.$ follows from $1.$ Finally, $5.$ is true as $\mu_{\intervalA{1}{N}}(\intervalA{1}{N}\setminus\intervalN{1}{M})\geq 1-\frac{M-1}{N}$ for all $N\in\positivenaturals$.
\end{proof}
	
\end{lem}
	\begin{lem}\label{partitiondenseN}
	Let $S\subseteq \positivenaturals$ and $a>1$. If there exists $C>0$ and $0<D<1$ such that $$\frac{\#S\cap \intervalN{a^n}{a^{n+1}}}{a^{n+1}-a^n}\geq 1-\frac{C}{a^{Dn}}$$ for all $n\in\N$, then $S$ has \fixdensity{D}, in particular, $S$ is \stardense.
	Furthermore, the converse holds as well, i.e., if $S$ has \fixdensity{D}, then there exists $C^\prime>0$ such that  $$\frac{\#S\cap \intervalN{a^n}{a^{n+1}}}{a^{n+1}-a^n}\geq 1-\frac{C^\prime}{a^{Dn}}$$ for all $n\in\N$.
\end{lem}
\begin{proof}
	
	Suppose that there exist $C>0$ and $0<D<1$ such that for all $n\in\N$ it holds that $$\frac{\#S\cap \intervalN{a^n}{a^{n+1}}}{a^{n+1}-a^n}\geq 1-\frac{C}{a^{Dn}}.$$ Let $N>0$. Look at the set $U=\positivenaturals\setminus S$. Then by assumption $$\#U\cap \intervalN{a^n}{a^{n+1}} \leq (a^{n+1}-a^n)\frac{C}{a^{Dn}}.$$ Let $n\in\N$ such that $a^n\leq N<a^{n+1}$. Then  $$\#U\cap \intervalA{1}{N} \leq\sum_{i=0}^{n}(a^{i+1}-a^i)\frac{C}{a^{Di}}=(a-1)C\sum_{i=0}^{n}a^{(1-D)i}=\frac{\left(a-1\right)C\left(a^{(1-D)(n+1)}-1\right)}{a^{1-D}-1}.$$
	Since $a^{n+1}\leq aN$, we obtain $$\#U\cap \intervalA{1}{N} \leq \frac{(a-1)Ca^{(1-D)(n+1)}}{a^{1-D}-1}\leq \frac{(a-1)Ca^{1-D}N^{1-D}}{a^{1-D}-1}.$$  Thus,  $$\frac{\#S\cap \intervalA{1}{N}}{N} \geq 1-\frac{(a-1)Ca^{1-D}}{\left(a^{1-D}-1\right)N^{D}}.$$ Thus, $S$ has \fixdensity{D}.
	
	Suppose now that $S$ has \density{C,D} for some $C>0$ and $0<D<1$. Look at the set $U=\positivenaturals\setminus S$. Then by hypothesis $$\frac{\#U\cap \intervalA{1}{a^{n+1}}}{a^{n+1}}\leq \frac{C}{a^{D(n+1)}}.$$ Thus, $$\frac{\#U\cap \intervalN{a^n}{a^{n+1}}}{a^{n+1}-a^n}\leq \frac{a^{n+1}}{a^{n+1}-a^n}\frac{C}{a^Da^{Dn}}= \frac{a^{1-D}}{a-1}\frac{C}{a^{Dn}},$$ therefore, $$\frac{\#S\cap\intervalN {a^n}{a^{n+1}}}{a^{n+1}-a^n}\geq 1- \frac{a^{1-D}}{a-1}\frac{C}{a^{Dn}}.$$ Hence, the claim follows with $C^\prime=\frac{a^{1-D}C}{a-1}$.
\end{proof}
We use Lemma \ref{Hoeffding} to prove the following lemma, which we will use in the proofs of Lemma \ref{beginning} and Lemma \ref{imagesareclose}.
\begin{lem}\label{paritystar}
	For every $\epsilon>0$ and $0<\alpha\leq 1$ the set $$\left\{m\in\positivenaturals\mid\forall \lfloor\alpha\log_2m\rfloor\leq k\leq\lfloor\log_2m\rfloor:-\epsilon k<  \sum_{i=0}^{k-1}p(m)_i-\frac{k}{2}< \epsilon k\right\}$$ is \stardense.
\end{lem}
\begin{proof}
	 For $0\leq k\leq n$ consider the set $$S^n_{k}=\left\{m\in\intervalN{2^n}{2^{n+1}}\mid\left| \sum_{i=0}^{k-1}p(m)_i-\frac{k}{2}\right|\geq \epsilon k\right\}.$$ By Lemma \ref{Hoeffding}, we have $\mu_{\intervalN{2^n}{2^{n+1}}}(S^n_{k})\leq 4e^{-2\epsilon^2k}$. Thus, for the set $$A^n=\bigcup^n_{k=\lfloor\alpha n\rfloor}S^n_{k}$$ we obtain $$\mu_{\intervalN{2^n}{2^{n+1}}}(A^n)\leq \sum^n_{k=\lfloor\alpha n\rfloor}4e^{-2\epsilon^2k}=4e^{-2\epsilon^2\lfloor\alpha n\rfloor}\frac{1-e^{-2\epsilon^2(n-\lfloor\alpha n\rfloor+1)}}{1-e^{-2\epsilon^2}}<4e^{2\epsilon^2}\frac{e^{-2\epsilon^2\alpha n}}{1-e^{-2\epsilon^2}}.$$
	 Hence, since $\lfloor\log_2m\rfloor=n$ for $m\in\intervalN{2^n}{2^{n+1}}$ and $\lfloor\alpha \log_2m\rfloor\geq \lfloor\alpha \lfloor \log_2m\rfloor\rfloor $, we obtain  \begin{align*}
	 	&\mu_{\intervalN{2^n}{2^{n+1}}}\left(\left\{m\in\intervalN{2^n}{2^{n+1}}\mid\forall \lfloor\alpha\log_2m\rfloor\leq k\leq\lfloor\log_2m\rfloor:-\epsilon k<  \sum_{i=0}^{k-1}p(m)_i-\frac{k}{2}< \epsilon k\right\}\right)\\
	 	&\geq 1- \frac{4e^{2\epsilon^2}}{1-e^{-2\epsilon^2}}e^{-2\epsilon^2\alpha n}. \end{align*} Since $$e^{-2\epsilon^2\alpha n}=\frac{1}{2^{(2\epsilon^2\alpha\log_2e) n}},$$ the claim follows from Lemma \ref{partitiondenseN} with $a=2$.
\end{proof}
	Now, we give an estimate of $r_k(m)$ on a \stardense\ab set.
	\begin{lem}\label{rest}
	Suppose that $\epsilon>0$. Then the set $$\left\{m\in\positivenaturals\mid\forall 0\leq k\leq \log_2m: r_k(m) 3^{\frac{k}{2}}m^{-\epsilon}<1\right\}$$ is \stardense.
\end{lem}
\begin{proof}
	Choose $0<\epsilon^\prime<\epsilon$.
	By 2. of Lemma \ref{splitrkm}, we get $$r_k(m)3^{\frac{k}{2}}m^{-\epsilon^\prime}< 3^{\frac{k}{2}}m^{-\epsilon^\prime}.$$
	Now, $3^{\frac{k}{2}}m^{-\epsilon^\prime}\leq 1$ holds as long as $3^{\frac{k}{2}}\leq m^{\epsilon^\prime}$, which is true for $k\log_2\sqrt{3}\leq \epsilon^\prime\log_2m$, or $k\leq \frac{\epsilon^\prime\log_2m}{\log_2\sqrt{3}}$. Set $\alpha=\frac{\epsilon^\prime}{\log_2\sqrt{3}}$. We have shown that (since $\epsilon^\prime<\epsilon$) 
	$$\left\{m\in\positivenaturals\mid\forall 0\leq k\leq \alpha\log_2m: r_k(m) 3^{\frac{k}{2}}m^{-\epsilon}<1\right\}=\positivenaturals.$$ 
		It remains to show that $$\left\{m\in\positivenaturals\mid\forall \alpha\log_2m\leq k\leq \log_2m: r_k(m) 3^{\frac{k}{2}}m^{-\epsilon}<1\right\}$$ is \stardense. 
	Choose $\delta>0$ with $\delta\log_23<\epsilon^\prime$ and set $$S=\left\{m\in\positivenaturals\mid\forall \lfloor\alpha\log_2m\rfloor\leq k\leq\lfloor\log_2m\rfloor:\left(\frac{1}{2}-\delta\right)k<  \sum_{i=0}^{k-1}p(m)_i\right\}.$$
	Note that $S$ is \stardense\ab by Lemma \ref{paritystar}.
 Suppose that $m\in S$ and set $k_0=\lfloor \alpha\log_2m\rfloor$. If $k_0\leq k\leq\log_2m$, then 
 \begin{equation}\label{rkmfirst}
 	\begin{split}
 	&r_k(m)3^{\frac{k}{2}}m^{-\epsilon^\prime}=\left(\sum_{i=0}^{k-1}\frac{p(m)_i}{3^{\sum_{j=0}^{i}p(m)_j}2^{k-i}}\right)3^{\frac{k}{2}}m^{-\epsilon^\prime}
	\\&\leq
	\left(\sum_{i=0}^{k_0-1}\frac{p(m)_i}{3^{\sum_{j=0}^{i}p(m)_j}2^{k_0-i}}\right)2^{k_0-k}3^{\frac{k}{2}}m^{-\epsilon^\prime}+\sum_{i=k_0}^{k-1}\frac{1}{3^{\frac{i+1}{2}(1-2\delta)}2^{k-i}}3^{\frac{k}{2}}m^{-\epsilon^\prime}\\
	&\leq r_{k_0}(m)3^{\frac{k_0}{2}}m^{-\epsilon^\prime}\frac{3^{\frac{k-k_0}{2}}}{2^{k-k_0}}+\sum_{i=k_0}^{k-1}\frac{2^i}{3^{\frac{i}{2}(1-2\delta)}}3^{\frac{k}{2}}2^{-k}m^{-\epsilon^\prime}
	\leq 1 + \frac{\left(\frac{2}{3^{\frac{1}{2}(1-2\delta)}}\right)^{k}}{\frac{2}{3^{\frac{1}{2}(1-2\delta)}}-1}3^{\frac{k}{2}}2^{-k}m^{-\epsilon^\prime}\\
	&= 1 + 3^{\delta k}\frac{1}{\frac{2}{3^{\frac{1}{2}(1-2\delta)}}-1}m^{-\epsilon^\prime}.
\end{split}\end{equation}  For the first inequality, we used that $p(m)_i\leq 1$ and $\left(\frac{1}{2}-\delta\right)(i+1)<  \sum_{j=0}^{i}p_j(m)$ for $k_0\leq i\leq k-1$, for the second inequality, we used $3^{\frac{i+1}{2}(1-2\delta)}>3^{\frac{i}{2}(1-2\delta)}$, and for the last inequality, we used $$r_{k_0}(m)3^{\frac{k_0}{2}}m^{-\epsilon^\prime}\leq 1,\ab\ab \frac{3^{\frac{k-k_0}{2}}}{2^{k-k_0}}\leq 1,$$ and the fact that $\sum_{i=k_0}^{k-1}q^i<\frac{q^k}{q-1}$ for $q=\frac{2}{3^{\frac{1}{2}(1-2\delta)}}>1$.

Now, since $\delta\log_23<\epsilon^\prime$  and $3^{\delta k}\leq m^{\delta\log_23}$ (since $k\leq\log_2m$), we conclude that \begin{equation}\label{rkmsecond}
1 + 3^{\delta k}\frac{1}{\frac{2}{3^{\frac{1}{2}(1-2\delta)}}-1}m^{-\epsilon^\prime}<2\end{equation} if $m$ is sufficiently large. Now, since $\epsilon^\prime<\epsilon$, combining (\ref{rkmfirst}) and (\ref{rkmsecond}), we obtain that $$r_k(m)3^{\frac{k}{2}}m^{-\epsilon}
	\leq 2m^{\epsilon^\prime-\epsilon}
	< 1$$ in case $m$ is sufficiently large. Thus, a co-finite\footnote{If $T$ is a subset of $S$, then $T$ is \textbf{co-finite in }$S$ if $S\setminus T$ is finite. In particular, if $S\subseteq\positivenaturals$ is \stardense, then $T$ is \stardense\ab as well.} subset of $S$ is contained in
	$$\left\{m\in\positivenaturals\mid\forall 0\leq k\leq \log_2m: r_k(m) 3^{\frac{k}{2}}m^{-\epsilon}<1\right\}.$$
	Hence, the proof is complete.
\end{proof}
	The next lemma gives a very coarse estimate of $\ColN^k(m)$, which we will sometimes use for small values of $k$.
	\begin{lem}\label{bruteforce}
		Suppose that $q\in\positivenaturals$, $k,p\in\N$, $A,B\in(0,\infty)$, and $A\leq \ColN^k(q)\leq B$. Then $\frac{A}{2^p}\leq\ColN^{k+p}(q)\leq 2^pB$. Furthermore, if $p\leq k$, then $\frac{A}{2^p}\leq\ColN^{k-p}(q)\leq 2^pB$. \end{lem}
	\begin{proof}
		Note that for every $Q\in\positivenaturals$, we have $\frac{Q}{2^p}\leq\ColN^p(Q)\leq2^PQ$ as  $\frac{3Q+1}{2}\leq 2Q$. But then $\frac{A}{2^p}\leq\frac{\ColN^k(q)}{2^p}\leq \ColN^{k+p}(q)\leq 2^p\ColN^k(q)\leq 2^pB$.
		 To see the second claim, observe that $A\leq\ColN^{k-p+p}(q)\leq 2^P\ColN^{k-p}(q)$, thus $\frac{A}{2^p}\leq \ColN^{k-p}(q)$, and $B\geq \ColN^{k-p+p}(q)\geq \frac{\ColN^{k-p}(q)}{2^p}$, thus $2^pB\geq\ColN^{k-p}(q)$.
	\end{proof}

\begin{lem}\label{beginning}
	Suppose that $\epsilon>0$. Then the set $$\left\{m\in\positivenaturals\mid\forall 0\leq k\leq\log_2m:  \left(\cons\right)^km^{1-\epsilon}\leq \ColN^k(m)\leq \left(\cons\right)^km^{1+\epsilon}\right\}$$ is \stardense.
\end{lem}
\begin{proof}

	First, note that by Lemma \ref{bruteforce}, we have	$$\ColN^k(m)\leq 2^km$$ for $m\in\positivenaturals.$ Now, $$2^km\leq \left(\cons\right)^km^{1+\epsilon}$$ is true as long as $4^k3^{-\frac{k}{2}}\leq m^{\epsilon}$ or $k\leq \frac{\epsilon}{2-\log_2\sqrt{3}}\log_2m$.\\
	Also, again  by Lemma \ref{bruteforce}, we have that $$\ColN^k(m)\geq \frac{1}{2^k}m,$$ and $$\frac{1}{2^k}m\geq \left(\cons\right)^km^{1-\epsilon}$$ is true as long as $3^{\frac{k}{2}}\leq m^{\epsilon}$ or $k\leq\frac{\epsilon}{\log_2\sqrt{3}}\log_2m$. Now, $$\min\left(\frac{\epsilon}{2-\log_2\sqrt{3}},\frac{\epsilon}{\log_2\sqrt{3}}\right)=\frac{\epsilon}{2-\log_2\sqrt{3}}.$$
	Set $\alpha=\frac{\epsilon}{2-\log_2\sqrt{3}}$. We have just shown that
	$$\left\{m\in\positivenaturals\mid\forall 0\leq k\leq\alpha\log_2m:  \left(\cons\right)^km^{1-\epsilon}\leq \ColN^k(m)\leq \left(\cons\right)^km^{1+\epsilon}\right\}=\positivenaturals.$$ Thus, it suffices to show that the set 
	$$\left\{m\in\positivenaturals\mid\forall \alpha\log_2m\leq k\leq\log_2m:  \left(\cons\right)^km^{1-\epsilon}\leq \ColN^k(m)\leq \left(\cons\right)^km^{1+\epsilon}\right\}$$ is \stardense.
 Choose $\eta>0$ such that $(\log_23+1)\eta<\epsilon$.
 By Lemma \ref{paritystar}, we know that the set 
  $$S=\left\{m\in\positivenaturals\mid\forall \alpha\log_2m\leq k\leq\log_2m:\left(\frac{1}{2}-\eta\right)k<  \sum_{i=0}^{k-1}p(m)_i< \left(\frac{1}{2}+\eta\right)k\right\}$$ is \stardense. By Lemma \ref{rest}, the set $$R=\left\{m\in\positivenaturals\mid\forall 0\leq k\leq \log_2m: r_k(m) 3^{\frac{k}{2}}m^{-\eta}<1\right\}$$ is \stardense\ab as well.
	Suppose that $m\in S\cap R$ and $\alpha\log_2m\leq k\leq\log_2m$.  Then, using Lemma \ref{ColIter}, we obtain \begin{equation}\label{beginningfirst}\begin{split}
		&\ColN^k(m)=\left(\frac{m}{2^k}+r_k(m)\right) 3^{\sum_{i=0}^{k-1}p(m)_i}\leq \left(\frac{m}{2^k}+r_k(m)\right) 3^{k\left(\frac{1}{2}+\eta\right)}\\
		&\leq \frac{m3^{k\left(\frac{1}{2}+\eta\right)}}{2^k}+m^{\eta}3^{k\eta}=\frac{m3^{\frac{k}{2}}}{2^k}3^{k\eta}+m^{\eta}3^{k\eta}.	\end{split}\end{equation} Furthermore, \begin{equation}\label{beginningsecond}
		m3^{k\eta}+2^k3^{-\frac{k}{2}}m^{\eta}3^{k\eta}\leq 3^{\eta\log_2m}(m+2^{\log_2m}m^{\eta})= m^{1+\eta\log_23}(1+m^{\eta}),
		\end{equation} where we used that $k\leq\log_2m$ and $3^{-\frac{k}{2}}\leq 1$. Thus, combining (\ref{beginningfirst}) and (\ref{beginningsecond}), we get
	$$\ColN^k(m)\leq\frac{m3^{\frac{k}{2}}}{2^k}3^{k\eta}+m^{\eta}3^{k\eta}\leq \left(\cons\right)^k(m3^{k\eta}+2^k3^{-\frac{k}{2}}m^{\eta}3^{k\eta})\leq\left(\cons\right)^km^{1+\eta\log_23}(1+m^{\eta}) .$$
	Now, $m^{\eta\log_23}(1+m^{\eta})<m^\epsilon$, in case $m$ is sufficiently large as we chose $\eta$ such that $$(\log_23+1)\eta<\epsilon.$$
	Similarly,
	\begin{align*}
		\ColN^k(m)=\left(\frac{m}{2^k}+r_k(m)\right) 3^{\sum_{i=0}^{k-1}p(m)_i}\geq \left(\frac{m}{2^k}+r_k(m)\right) 3^{k\left(\frac{1}{2}-\eta\right)}\geq \frac{m3^{k\left(\frac{1}{2}-\eta\right)}}{2^k}	\end{align*} and $\frac{m3^{k\left(\frac{1}{2}-\eta\right)}}{2^k}\geq \frac{3^{\frac{k}{2}}}{2^k}m^{1-\epsilon}$ since $m^{\epsilon}\geq m^{\eta\log_23}= 3^{\eta\log_2m}\geq  3^{\eta k} $, which holds since $k\leq \log_2m$ and $\eta\log_23<\eta(\log_23+1)<\epsilon.$
Thus, a co-finite subset of $S\cap R$ is contained in $$\left\{m\in\positivenaturals\mid\forall 0\leq k\leq\log_2m:  \left(\cons\right)^km^{1-\epsilon}\leq \ColN^k(m)\leq \left(\cons\right)^km^{1+\epsilon}\right\}.$$
Hence, the proof is complete.
\end{proof}

The following lemma is crucial for the iterative argument in Lemma \ref{Iterate}.
	\begin{lem}\label{leverage}
		Suppose that $0<D\leq 1$, $K\in\positivenaturals$, and $S\subseteq\positivenaturals$ has $\fixdensity{D}$. Then, for every $0<\eta<\frac{D}{2-D}$, the set $$\left\{m\in\positivenaturals\mid \forall 0\leq k\leq \lfloor\eta \log_3m\rfloor\forall 0\leq i< K\cdot3^k:\ab 3^km+i\in S\right\}$$ has \fixdensity{(D(1+\eta)-2\eta)}.
	\end{lem} 
	\begin{proof}
		Choose $\eta$ such that $0<\eta<\frac{D}{2-D}$. We will show that for the set $$U^n=\left\{m\in\intervalN{3^n}{3^{n+1}}\mid \forall 0\leq k\leq \lfloor\eta \log_3m\rfloor\forall 0\leq i< K\cdot3^k:\ab 3^km+i\in S\right\}$$ we have $\mu_{\intervalN{3^n}{3^{n+1}}}(U^n)\geq 1-\frac{C_0}{3^{n(D(1+\eta)-2\eta)}}$ for some $C_0>0$,	which implies the claim by Lemma \ref{partitiondenseN} with $a=3$.	Define $R=\positivenaturals\setminus S$ and $R_M=\intervalA{1}{M}\cap R$.
		By hypothesis, we have $\frac{\#R_M}{M}\leq \frac{C}{M^D}$ for all $M\in\positivenaturals$ and some $C>0$.
		
		 For $ k\leq {\lfloor\eta (n+1)\rfloor}$ and $0\leq i< K\cdot 3^k$, consider the set $Q^n_{k,i}=\left\{m\in\intervalN{3^n}{3^{n+1}}\mid m\cdot 3^k+i\in R\right\}$. Since the map from $\intervalN{3^n}{3^{n+1}}$ to $\intervalN{1}{(3^k3^{n+1}+i)}$ that sends $m$ to $m\cdot 3^k+i$ is injective, we have $\#Q^n_{k,i}\leq \#R_{3^{n+1}3^k+i}$. Thus,		
		
		$$\frac{\#Q^n_{k,i}}{3^{n+1}-3^n}
		\leq
		\frac{\#R_{3^{n+1}3^k+i}}{2\cdot 3^{n}} \leq
		\frac{C(3^{n+1}3^{k}+i)}{2\cdot 3^{n}(3^{n+1}3^{k}+i)^D}
		=
		\frac{C(3^{n+1}3^{k}+i)^{1-D}}{2\cdot 3^{n}}\leq \frac{C3^{(1-D)k}(3^{n+1}+K)^{1-D}}{2\cdot 3^{n}}.$$
		For the last inequality, we used that $0\leq i< K\cdot 3^k$.
		Hence,
		\begin{align*}
			&\frac{\#\bigcup_{k=0}^{ \lfloor\eta (n+1)\rfloor }\bigcup_{i=0}^{ K\cdot 3^k-1}{Q^n_{k,i}}}{2\cdot 3^{n}}
			\leq
			\sum_{k=0}^{\lfloor\eta (n+1)\rfloor }\ab\ab\sum_{i=0}^{K\cdot 3^k-1} \frac{C3^{(1-D)k}(3^{n+1}+K)^{1-D}}{2\cdot 3^{n}}\\
			&=
		\frac{C(3^{n+1}+K)^{1-D}}{2\cdot 3^{n}}\sum_{k=0}^{ \lfloor\eta (n+1)\rfloor }\ab\ab\sum_{i=0}^{K\cdot 3^k-1} 3^{(1-D)k}
			=
			K\frac{C(3^{n+1}+K)^{1-D}}{2\cdot 3{n}}\sum_{k=0}^{\lfloor\eta (n+1)\rfloor } 3^{(2-D)k}\\ 
			&<
			K\frac{C(3^{n+1}+K)^{1-D}}{2\cdot 3^{n}}\frac{3^{(2-D)(\eta n+\eta+1)}}{3^{2-D}-1}
			=
				K\frac{C(3^{n+1}+K)^{1-D}}{2\cdot 3^{n}}\frac{3^{(2-D)(1+\eta)}3^{n(2-D)\eta}}{3^{2-D}-1}\\
				&
				=K\frac{C(3+\frac{K}{3^{n}})^{1-D}}{2}\frac{3^{n(1-D)}}{3^{n}}\frac{3^{(2-D)(1+\eta)}3^{n(2-D)\eta}}{3^{2-D}-1}
					=\frac{K}{2}\frac{C(3+\frac{K}{3^{n}})^{1-D}3^{(2-D)(1+\eta)}}{3^{2-D}-1}3^{n((2-D)\eta-D)}\\
					 &\leq\frac{K}{2}\frac{C(3+K)^{1-D}3^{(2-D)(1+\eta)}}{3^{2-D}-1}3^{n(2\eta-D(1+\eta))}.\\
		\end{align*}
		Where the second inequality follows from $$\sum_{k=0}^{\lfloor\eta (n+1)\rfloor } 3^{(2-D)k}=\frac{3^{(2-D)(\lfloor\eta (n+1)\rfloor+1)}-1}{ 3^{2-D}-1}<\frac{3^{(2-D)(\eta (n+1)+1)}}{ 3^{2-D}-1}.$$ 
			Put $$B^n=\intervalN{3^n}{3^{n+1}}\setminus \bigcup_{k=0}^{\lfloor\eta (n+1)\rfloor }\bigcup_{i=0}^{K\cdot 3^k-1}{Q^n_{k,i}}.$$ Then, $\mu_{\intervalN{3^n}{3^{n+1}}}(B^n)\geq 1-\frac{C_0}{3^{n(D(1+\eta)-2\eta)}}$ for $C_0=\frac{K}{2}\frac{C(3+K)^{1-D}3^{(2-D)(1+\eta)}}{3^{2-D}-1}$. Furthermore, $B^n\subseteq U^n$ as $\lfloor\eta\log_3m\rfloor\leq \lfloor\eta (n+1)\rfloor$ for all $m\in\intervalN{3^n}{3^{n+1}}$. Thus, the claim follows from Lemma \ref{partitiondenseN} with $a=3$.
	\end{proof}
\begin{lem}\label{scale}
	Suppose that $\alpha\in(0,\log_32]$, $0<D\leq 1$ and $S\subseteq\positivenaturals$ has \fixdensity{D}. Then $$\left\{m\in\positivenaturals\mid \left\lfloor\frac{3^{\lfloor\alpha\lfloor\log_2m\rfloor\rfloor}}{2^{\lfloor\log_2m\rfloor}}m\right\rfloor\in S\right\}$$ has \fixdensity{D\alpha\log_23}.
\end{lem}
\begin{proof}
		Choose $C>0$ such that $S$ has \density{C,D}. Put $R=\positivenaturals\setminus S$.  Consider the map $$\Phi:\intervalN{2^n}{2^{n+1}}\rightarrow\intervalN{{3^{\lfloor\alpha n\rfloor}}}{{2\cdot3^{\lfloor\alpha n\rfloor}}};\ab m\mapsto \left\lfloor\frac{3^{\lfloor\alpha n\rfloor}}{2^{n}}m\right\rfloor.$$ Then, for any $M\in \intervalN{{3^{\lfloor\alpha n\rfloor}}}{{2\cdot3^{\lfloor\alpha n\rfloor}}}$ we have
		$$\Phi(m)=M\Leftrightarrow M\leq \frac{3^{\lfloor \alpha n\rfloor}}{2^n}m< M+1\Leftrightarrow \frac{2^nM}{3^{\lfloor \alpha n\rfloor}}\leq m< \frac{2^nM}{3^{\lfloor \alpha n\rfloor}}+\frac{2^n}{3^{\lfloor \alpha n\rfloor}} .$$
		Thus, $\Phi^{-1}(\left\{M\right\})=\left[\frac{2^n}{3^{\lfloor\alpha n\rfloor}}M,M\frac{2^n}{3^{\lfloor\alpha n\rfloor}}+\frac{2^n}{3^{\lfloor\alpha n\rfloor}}\right)\cap\N$. Hence, $\#\Phi^{-1}(\left\{M\right\})\leq \left\lceil\frac{2^n}{3^{\lfloor\alpha n\rfloor}}\right\rceil$. We conclude that \begin{align*}
			&\#\left\{m\in\intervalN{2^n}{2^{n+1}}\mid \left\lfloor\frac{3^{\lfloor\alpha n\rfloor}}{2^{n}}m\right\rfloor\in R\right\}\\&= \#\left\{m\in\intervalN{2^n}{2^{n+1}}\mid \left\lfloor\frac{3^{\lfloor\alpha n\rfloor}}{2^{n}}m\right\rfloor\in R\cap \intervalN{1}{2\cdot3^{\lfloor\alpha n\rfloor}}\right\}\\
			&\leq\left\lceil\frac{2^n}{3^{\lfloor\alpha n\rfloor}}\right\rceil\#R\cap \intervalN{1}{2\cdot3^{\lfloor\alpha n\rfloor}}\leq \left\lceil\frac{2^n}{3^{\lfloor\alpha n\rfloor}}\right\rceil\frac{C}{(2\cdot3^{\lfloor\alpha n\rfloor})^D}2\cdot3^{\lfloor\alpha n\rfloor}\\
			&\leq \left(\frac{2^n}{3^{\lfloor\alpha n\rfloor}}+1\right)\frac{C}{(2\cdot3^{\alpha n-1})^D}2\cdot3^{\lfloor\alpha n\rfloor}\leq (2^n+3^{\alpha n})\frac{3^D2^{1-D}C}{2^{D\alpha n\log_23}}.\end{align*} Thus, $$\frac{\#\left\{m\in\intervalN{2^n}{2^{n+1}}\mid \left\lfloor\frac{3^{\lfloor\alpha n\rfloor}}{2^{n}}m\right\rfloor\in R\right\}}{2^{n+1}-2^n}\leq \frac{2^n+3^{\alpha n}}{2^n}\frac{3^D2^{1-D}C}{2^{D\alpha n\log_23}}\leq 2\frac{3^D2^{1-D}C}{2^{nD\alpha \log_23}}. $$ For the last inequality, we used that $3^{\alpha n}\leq 2^n$ as $0<\alpha\leq\log_32$.
		Using Lemma \ref{partitiondenseN} with $a=2$, we conclude that the set $$\left\{m\in\positivenaturals\mid \left\lfloor\frac{3^{\lfloor\alpha\lfloor\log_2m\rfloor\rfloor}}{2^{\lfloor\log_2m\rfloor}}m\right\rfloor\in S\right\}$$ has \fixdensity{D\alpha\log_23}.
\end{proof}
	\begin{lem}\label{imagesareclose}
		Suppose that $\eta>0$. Then the set $$\left\{m\in\positivenaturals\mid \exists\ab 0\leq L\leq 2\eta \lfloor\log_2m\rfloor\exists\ab  0\leq i< 2\cdot 3^L: \ColN^{\lfloor\log_2m\rfloor}(m)=\left\lfloor\frac{3^{\left\lfloor\left(\frac{1}{2}-\eta\right)\lfloor\log_2m\rfloor\right\rfloor}}{2^{\lfloor\log_2m\rfloor}}m\right\rfloor3^L+i\right\}$$ is \stardense.
	\end{lem}
\begin{proof}
	Let $\delta>0$ and $0<\eta^\prime<\eta$. By Lemma \ref{rest}, we know that the set $$\left\{m\in\positivenaturals\mid\forall 0\leq k\leq \log_2m: r_k(m) 3^{\frac{k}{2}}m^{-\delta}<1\right\}$$ is \stardense. In particular, the set $$\left\{m\in\positivenaturals\mid r_{\lfloor\log_2m\rfloor}(m)\cdot 3^{\frac{\lfloor\log_2m\rfloor}{2}}m^{-\delta}<1\right\}$$ is \stardense\ab as well. If we choose $\delta>0$ small enough, we can ensure that $$3^{\frac{\lfloor\log_2m\rfloor}{2}}m^{-\delta}> 3^{\left\lfloor\left(\frac{1}{2}-\eta\right)\lfloor\log_2m\rfloor\right\rfloor}$$ for sufficiently large $m$. Thus, $$S=\left\{m\in\positivenaturals\mid r_{\lfloor\log_2m\rfloor}(m)\cdot 3^{\left\lfloor\left(\frac{1}{2}-\eta\right)\lfloor\log_2m\rfloor\right\rfloor}<1\right\}$$ is also \stardense. By Lemma \ref{paritystar}, we know that the set $$Q=\left\{m\in\positivenaturals\mid\left(\frac{1}{2}-\eta\right)\lfloor\log_2 m\rfloor\leq  \sum_{k=0}^{\lfloor\log_2m\rfloor-1}p(m)_k\leq \left(\frac{1}{2}+\eta^\prime\right)\lfloor\log_2m\rfloor\right\}$$ is \stardense.
	Let $m\in S\cap Q$. By Lemma \ref{ColIter}, we obtain \begin{equation}\label{imagesareclose1}
	\ColN^{\lfloor\log_2 m\rfloor}(m)=\left(\frac{m}{2^{\lfloor\log_2m\rfloor}}+r_{\lfloor\log_2m\rfloor}(m)\right)\cdot 3^{\sum_{k=0}^{\lfloor\log_2m \rfloor-1}p(m)_k}.\end{equation} Now, since $m\in Q$ we have $${\left(\frac{1}{2}-\eta\right)\lfloor\log_2m \rfloor}\leq {\sum_{k=0}^{\lfloor\log_2m \rfloor-1}p(m)_k}\leq {\left(\frac{1}{2}+\eta^\prime\right)\lfloor\log_2m \rfloor}.$$ We conclude that $$0\leq \sum_{k=0}^{\lfloor\log_2m\rfloor-1}p(m)_k-\left\lfloor\left(\frac{1}{2}-\eta\right)\lfloor\log_2m\rfloor\right\rfloor < \left(\eta+\eta^\prime\right)\lfloor\log_2m\rfloor+1.$$
	Thus, we have $$\sum_{k=0}^{\lfloor\log_2m\rfloor-1}p(m)_k=L+\left\lfloor\left(\frac{1}{2}-\eta\right)\lfloor\log_2m\rfloor\right\rfloor$$ for some $0\leq L\leq\left(\eta+\eta^\prime\right) \lfloor\log_2m\rfloor+1\leq 2\eta \lfloor\log_2m\rfloor$, in case $m$ is sufficiently large, as $\eta^\prime<\eta$. Hence, with (\ref{imagesareclose1}) we get \begin{align*}
		&\ColN^{\lfloor\log_2m\rfloor}(m)=\left(\frac{m}{2^{\lfloor\log_2m\rfloor}}+r_{\lfloor\log_2m\rfloor}(m)\right)\cdot 3^{\left\lfloor\left(\frac{1}{2}-\eta\right)\lfloor\log_2m \rfloor\right\rfloor+L}\\&=\left\lfloor\frac{3^{\left\lfloor\left(\frac{1}{2}-\eta\right)\lfloor\log_2m\rfloor\right\rfloor}}{2^{\lfloor\log_2m\rfloor}}m\right\rfloor3^L+\left(\frac{3^{\left\lfloor\left(\frac{1}{2}-\eta\right)\lfloor\log_2m\rfloor\right\rfloor}}{2^{\lfloor\log_2m\rfloor}}m-\left\lfloor\frac{3^{\left\lfloor(\frac{1}{2}-\eta)\lfloor\log_2m\rfloor\right\rfloor}}{2^{\lfloor\log_2m\rfloor}}m\right\rfloor\right)3^L\\&+r_{\lfloor\log_2m\rfloor}(m)\cdot 3^{\left\lfloor\left(\frac{1}{2}-\eta\right)\lfloor\log_2m\rfloor\right\rfloor}3^L=\left\lfloor\frac{3^{\left\lfloor(\frac{1}{2}-\eta)\lfloor\log_2m\rfloor\right\rfloor}}{2^{\lfloor\log_2m\rfloor}}m\right\rfloor3^L+i\end{align*}
	 with $0\leq i< 2\cdot3^{L}$ as $$0\leq\frac{3^{\left\lfloor\left(\frac{1}{2}-\eta\right)\lfloor\log_2m\rfloor\right\rfloor}}{2^{\lfloor\log_2m\rfloor}}m-\left\lfloor\frac{3^{\left\lfloor\left(\frac{1}{2}-\eta\right)\lfloor\log_2m\rfloor\right\rfloor}}{2^{\lfloor\log_2m\rfloor}}m\right\rfloor<1$$ and, as $m\in S$, $$0\leq r_{\lfloor\log_2m\rfloor}(m)\cdot 3^{\left\lfloor(\frac{1}{2}-\eta)\lfloor\log_2m\rfloor\right\rfloor}<1.$$ Furthermore, $i\in\N$ as $\ColN^{\lfloor\log_2m\rfloor}(m)\in\N$ and $\left\lfloor\frac{3^{\left\lfloor\left(\frac{1}{2}-\eta\right)\lfloor\log_2m\rfloor\right\rfloor}}{2^{\lfloor\log_2m\rfloor}}m\right\rfloor3^L\in\N$.
	 
	 Since $S\cap Q$ is \stardense, the claim follows.
\end{proof}
Now we turn to the most crucial lemma for the proof of the Main Theorem.
	\begin{lem}\label{Iterate}
		Suppose that $P\subseteq\positivenaturals$ is \stardense. Then the set $$\left\{m\in\positivenaturals\mid \ColN^{\lfloor\log_2m\rfloor}(m)\in P\right\}$$ is \stardense\ab as well.
	\end{lem}
	\begin{proof}

	 Choose $0<D<1$ such that $P$ has \fixdensity{D}. Also choose $0<\eta^\prime<\frac{D}{2-D}$ and let $$P^\prime= \left\{m\in\positivenaturals\mid \forall k\leq \lfloor\eta^\prime \log_3m\rfloor\forall 0\leq i< 2\cdot3^k:\ab 3^km+i\in P\right\}.$$ Applying Lemma \ref{leverage} with $K=2$, we obtain that $P^\prime$ is \stardense. Choose $\eta>0$ such that \begin{equation}\label{interatefirst}
	 2\eta \lfloor\log_2m\rfloor\leq \left\lfloor\eta^\prime\log_3 \left\lfloor\frac{3^{\left\lfloor\left(\frac{1}{2}-\eta\right)\lfloor\log_2m\rfloor\right\rfloor}}{2^{\lfloor\log_2m\rfloor}}m\right\rfloor\right\rfloor\end{equation} for all sufficiently large $m\in\positivenaturals$. This is possible as for all sufficiently large $m$ we have (using that $\log_3(M-1)>\log_3M-1$ for all $M\geq 2$)
	 $$\left\lfloor\eta^\prime\log_3 \left\lfloor\frac{3^{\left\lfloor\left(\frac{1}{2}-\eta\right)\lfloor\log_2m\rfloor\right\rfloor}}{2^{\lfloor\log_2m\rfloor}}m\right\rfloor\right\rfloor\geq \eta^\prime\log_3 \left(\frac{3^{\left\lfloor\left(\frac{1}{2}-\eta\right)\lfloor\log_2m\rfloor\right\rfloor}}{2^{\log_2m}}m\right)-\eta^\prime-1\geq \eta^\prime\left(\frac{1}{2}-\eta\right)\lfloor\log_2m\rfloor-2\eta^\prime-1 $$ and
	 $$2\eta\lfloor\log_2m\rfloor\leq\eta^\prime\left(\frac{1}{2}-\eta\right)\lfloor\log_2m\rfloor-2\eta^\prime -1$$  for sufficiently large $m$ if $0<\eta<\frac{\eta^\prime}{2\left(\eta^\prime+2\right)}$.
	  Put $$P^{\prime\prime}=\left\{m\in\positivenaturals\mid \left\lfloor\frac{3^{\left\lfloor\left(\frac{1}{2}-\eta\right)\lfloor\log_2m\rfloor\right\rfloor}}{2^{\lfloor\log_2m\rfloor}}m\right\rfloor\in P^{\prime}\right\}.$$ Applying Lemma \ref{scale} with $\alpha=\frac{1}{2}-\eta$, we get that
	$P^{\prime\prime}$ is \stardense. 
	By Lemma \ref{imagesareclose}, the set \begin{equation*}
		S=\left\{m\in\positivenaturals\mid \exists\ab 0\leq k\leq 2\eta \lfloor\log_2m\rfloor\exists\ab  0\leq i< 2\cdot 3^k: \ColN^{\lfloor\log_2m\rfloor}(m)=\left\lfloor\frac{3^{\left\lfloor\left(\frac{1}{2}-\eta\right)\lfloor\log_2m\rfloor\right\rfloor}}{2^{\lfloor\log_2m\rfloor}}m\right\rfloor3^k+i\right\}\end{equation*} is \stardense. Let $m\in S\cap P^{\prime\prime}$.	Since $m\in S$,  we have that 	$$\ColN^{\lfloor\log_2m\rfloor}(m)=\left\lfloor\frac{3^{\left\lfloor\left(\frac{1}{2}-\eta\right)\lfloor\log_2m\rfloor\right\rfloor}}{2^{\lfloor\log_2m\rfloor}}m\right\rfloor3^{L_0}+j_0$$ for some $0\leq L_0\leq 2\eta \lfloor\log_2m\rfloor$ and $0\leq j_0< 2\cdot3^{L_0}$. Since $m\in P^{\prime\prime}$, we conclude that $$\left\lfloor\frac{3^{\left\lfloor\left(\frac{1}{2}-\eta\right)\lfloor\log_2m\rfloor\right\rfloor}}{2^{\lfloor\log_2m\rfloor}}m\right\rfloor\in P^{\prime}.$$ Thus, by definition of $P^{\prime}$ we have
	$$\left\lfloor\frac{3^{\left\lfloor\left(\frac{1}{2}-\eta\right)\lfloor\log_2m\rfloor\right\rfloor}}{2^{\lfloor\log_2m\rfloor}}m\right\rfloor3^k+i\in P$$ for all $0\leq k\leq\left\lfloor\eta^\prime\log_3\left\lfloor\frac{3^{\left\lfloor\left(\frac{1}{2}-\eta\right)\lfloor\log_2m\rfloor\right\rfloor}}{2^{\lfloor\log_2m\rfloor}}m\right\rfloor\right\rfloor$ and $0\leq i\leq 2\cdot3^k$. In particular, 
	 $$\ColN^{\lfloor\log_2m\rfloor}(m)=\left\lfloor\frac{3^{\left\lfloor\left(\frac{1}{2}-\eta\right)\lfloor\log_2m\rfloor\right\rfloor}}{2^{\lfloor\log_2m\rfloor}}m\right\rfloor3^{L_0}+j_0\in P$$ as by (\ref{interatefirst}) we have $$L_0\leq 2\eta \lfloor\log_2m\rfloor\leq \eta^\prime\left\lfloor\log_3 \left\lfloor\frac{3^{\left\lfloor\left(\frac{1}{2}-\eta\right)\lfloor\log_2m\rfloor\right\rfloor}}{2^{\lfloor\log_2m\rfloor}}m\right\rfloor \right\rfloor. $$ Thus, the proof is complete as $S\cap P^{\prime\prime}$ is \stardense.
	\end{proof}
	We need one more lemma before we are going to prove Theorem \ref{main}.
	\begin{lem}\label{longtermparity}
		Suppose that $\lambda\in [0,1]$. If the set  \begin{align*}
		&\Biggl\{m\in\positivenaturals\mid\forall 0\leq k\leq \frac{1-\lambda}{1-\log_2\sqrt{3}}\log_2m:  \left(\cons\right)^km^{1-\epsilon}\leq \ColN^k(m)\leq \left(\cons\right)^km^{1+\epsilon}\\&\wedge r_k(m)3^{\frac{k}{2}}m^{-\epsilon}<1\Biggr\}\end{align*} is \stardense\ab for every $\epsilon>0$, then the set  $$\left\{m\in\positivenaturals\mid\forall 0\leq k\leq \frac{1-\lambda}{1-\log_2\sqrt{3}}\log_2m:  \frac{k}{2}-\epsilon\log_2m\leq\sum_{i=0}^{k-1}p(m)_i\leq \frac{k}{2}+\epsilon\log_2m\right\}$$ is \stardense\ab for every $\epsilon>0$.
	\end{lem}
\begin{proof}
Suppose that $\lambda\in[0,1]$. Let $\epsilon>0$ and suppose that \begin{align*}
&S_\delta=\\&\Biggl\{m\in\positivenaturals\mid\forall 0\leq k\leq \frac{1-\lambda}{1-\log_2\sqrt{3}}\log_2m:  \left(\cons\right)^km^{1-\delta}\leq \ColN^k(m)\leq\left(\cons\right)^km^{1+\delta}\\&\wedge r_k(m)3^{\frac{k}{2}}m^{-\delta}<1\Biggr\}\end{align*} is \stardense\ab for every $\delta>0$. Suppose that $m\in S_\delta$. Then, using Lemma \ref{ColIter}, we get
$$\frac{m}{2^{k}}3^{\sum_{i=0}^{k-1}p(m)_i}\leq \left(\frac{m}{2^{k}}+r_k(m)\right)3^{\sum_{i=0}^{k-1}p(m)_i}=\ColN^{k}(m)\leq \left(\cons\right)^{k}m^{1+\delta} $$
for every $0\leq k\leq \frac{1-\lambda}{1-\log_2\sqrt{3}}\log_2m$.
Thus, $3^{\sum_{i=0}^{k-1}p(m)_i}\leq 3^{\frac{k}{2}}m^{\delta}$ and therefore \begin{equation}\label{longtermparityfourth}
\sum_{i=0}^{k-1}p(m)_i\leq \frac{k}{2}+\delta\log_3m=\frac{k}{2}+\delta(\log_32)\log_2m.\end{equation}
Similarly, 
\begin{equation*}\left(\cons\right)^{k}m^{1-\delta}\leq\ColN^{k}(m)=\left(\frac{m}{2^{k}}+r_k(m)\right)3^{\sum_{i=0}^{k-1}p(m)_i}\leq \left(\frac{m}{2^{k}}+m^{\delta} 3^{-\frac{k}{2}}\right)3^{\sum_{i=0}^{k-1}p(m)_i}.\end{equation*}
Thus,
\begin{equation}\label{longtermparitysecond}3^{\frac{k}{2}}m^{-\delta}\leq \left(1+2^km^{-1+\delta} 3^{-\frac{k}{2}}\right)3^{\sum_{i=0}^{k-1}p(m)_i}.\end{equation}
Now, since $k\leq \frac{\log_2m}{1-\log_2\sqrt{3}}$, we obtain \begin{equation}\label{longtermparitythird}
m^{-1+\delta} \left(\frac{2}{\sqrt{3}}\right)^k\leq m^{-1+\delta} \left(\frac{2}{\sqrt{3}}\right)^{\frac{\log_2 m}{1-\log_2\sqrt{3}}} =m^{-1+\delta}\cdot m=m^\delta.\end{equation} Hence, combining (\ref{longtermparitysecond}) and (\ref{longtermparitythird}) we get
$$3^{\frac{k}{2}}m^{-\delta}\leq \left(1+2^km^{-1+\delta} 3^{-\frac{k}{2}}\right)3^{\sum_{i=0}^{k-1}p(m)_i} \leq \left(1+m^\delta\right)3^{\sum_{i=0}^{k-1}p(m)_i}\leq 2\cdot m^\delta3^{\sum_{i=0}^{k-1}p(m)_i}.$$
Thus, 
$$\frac{1}{2}m^{-2\delta}3^{\frac{k}{2}}\leq 3^{\sum_{i=0}^{k-1}p(m)_i},$$ or
$$\frac{k}{2}-2\delta\log_3m-\log_32\leq \sum_{i=0}^{k-1}p(m)_i.$$ Thus, if $2\delta\log_32<\epsilon$ and $m$ is sufficiently large, together with (\ref{longtermparityfourth}), we obtain $$\frac{k}{2}-\epsilon\log_2m\leq\sum_{i=0}^{k-1}p(m)_i\leq \frac{k}{2}+\epsilon\log_2m.$$
\end{proof}
	Now, we have all the tools to prove Theorem \ref{main}. We will prove a slightly stronger version, that bounds $r_k(m)$ as well.
	\begin{thm}\label{secondmaintheorem}
		Suppose that $\epsilon>0$. Then the set \begin{align*}
		&\left\{m\in\positivenaturals\mid\forall 0\leq k\leq \frac{\log_2m}{1-\log_2\sqrt{3}}:  \left(\cons\right)^km^{1-\epsilon}\leq \ColN^k(m)\leq \left(\cons\right)^km^{1+\epsilon}\wedge r_k(m)3^{\frac{k}{2}}m^{-\epsilon}<1\right\}\end{align*} is \stardense.
	\end{thm}
	\begin{proof}
		First, we show that it is enough to show that the set  \begin{align*}
			&S^\lambda_\epsilon=\Biggl\{m\in\positivenaturals\mid\forall 0\leq k\leq \frac{1-\lambda}{1-\log_2\sqrt{3}}\log_2m:  \left(\cons\right)^km^{1-\epsilon}\leq \ColN^k(m)\leq \left(\cons\right)^km^{1+\epsilon}\wedge\\& r_k(m)3^{\frac{k}{2}}m^{-\epsilon}<1\Biggr\}\end{align*} is \stardense\ab for all $1\geq\lambda> 0$ and $\epsilon>0$ .
		To see this, fix $\epsilon>0$ and suppose that $S^\lambda_\delta$ is \stardense\ab for all $\delta>0$ and for all $\lambda\in(0,1]$. Assume that $m\in S^\lambda_\delta$ and $$\frac{1-\lambda}{1-\log_2\sqrt{3}}\log_2m<k\leq \frac{\log_2m}{1-\log_2\sqrt{3}}.$$ Set $l_0=\left\lfloor\frac{1-\lambda}{1-\log_2\sqrt{3}}\log_2m\right\rfloor$. Since $\left(\cons\right)^{l_0}m^{1-\delta}\leq \ColN^{l_0}(m)\leq \left(\cons\right)^{l_0}m^{1+\delta}$, we obtain using Lemma \ref{bruteforce} $$\ColN^{l_0+(k-l_0)}(m)=\ColN^{k-l_0}(\ColN^{l_0}(m))\geq\frac{1}{2^{k-l_0}} \left(\cons\right)^{l_0}m^{1-\delta}=\frac{1}{3^{\frac{1}{2}(k-l_0)}} \left(\cons\right)^{k}m^{1-\delta}$$ and $$\ColN^{l_0+(k-l_0)}(m)=\ColN^{k-l_0}(\ColN^{l_0}(m))\leq 2^{k-l_0} \left(\cons\right)^{l_0}m^{1+\delta}=\left(\frac{4}{\sqrt{3}}\right)^{k-l_0} \left(\cons\right)^{k}m^{1+\delta}.$$ Thus, $$\frac{1}{3^{\frac{1}{2}(k-l_0)}} \left(\cons\right)^{k}m^{1-\delta}\leq\ColN^{k}(m)\leq\left(\frac{4}{\sqrt{3}}\right)^{k-l_0} \left(\cons\right)^{k}m^{1+\delta}.$$ We can ensure that  $\frac{1}{3^{\frac{1}{2}(k-l_0)}} m^{1-\delta}\geq m^{1-\epsilon}$ and $\left(\frac{4}{\sqrt{3}}\right)^{k-l_0} m^{1+\delta}\leq m^{1+\epsilon}$. To see that this is possible, just note that as $k\leq \frac{\log_2m}{1-\log_2\sqrt{3}}$ and $l_0\geq \frac{1-\lambda}{1-\log_2\sqrt{3}}\log_2m-1$, we obtain $$\frac{k-l_0}{\log_2m}\leq  \frac{1}{1-\log_2\sqrt{3}}- \frac{1-\lambda}{1-\log_2\sqrt{3}}+\frac{1}{\log_2m},$$ which goes to $0$ as $\lambda$ goes to $0$ and $m$ goes to $+\infty$. Thus, if $\delta,\lambda>0$ are small enough, then any sufficiently large $m\in S^{\lambda}_\delta$ fulfills $$\forall 0\leq k\leq \frac{\log_2m}{1-\log_2\sqrt{3}}:  \left(\cons\right)^km^{1-\epsilon}\leq \ColN^k(m)\leq \left(\cons\right)^km^{1+\epsilon}.$$
		Now, we will estimate $r_k(m)$ for $l_0<k\leq\frac{\log_2m}{1-\log_2\sqrt{3}}$. By Lemma \ref{longtermparity}, the set $$P_{\eta}^\lambda=\left\{m\in\positivenaturals\mid\forall 0\leq k\leq \frac{1-\lambda}{1-\log_2\sqrt{3}}\log_2m:  \frac{k}{2}-\eta\log_2m\leq\sum_{i=0}^{k-1}p(m)_i\right\}$$ is also \stardense\ab for every $\eta>0$. Assume that $m\in S^\lambda_\delta\cap P_\eta^\lambda$. Note that (using 1. of Lemma \ref{splitrkm})
		\begin{equation}\label{maintheoremthird}\begin{split}
				&r_k(m)=	r_{k-l_0+l_0}(m)
		=2^{-k+l_0}r_{l_0}(m)+\frac{1}{3^{\sum_{j=0}^{l_0-1}p(m)_j}}r_{k-l_0}(\ColN^{l_0}(m))\\
		&
		\leq 2^{l_0-k}3^{-\frac{l_0}{2}}m^{\delta}+\frac{1}{3^{\sum_{j=0}^{l_0-1}p(m)_j}}.\end{split}\end{equation} We used that $m\in S^\lambda_\delta$ and $r_l(M)\leq 1$ for all $l\in\N$ and $M\in\positivenaturals$  (by 2. of Lemma \ref{splitrkm}). Note that \begin{equation}\label{maintheoremfourth}
2^{l_0-k}3^{-\frac{l_0}{2}}m^{\delta}3^{\frac{k}{2}}=\frac{3^{\frac{k-l_0}{2}}}{2^{k-l_0}}m^\delta\leq m^\delta<\frac{1}{2}m^\epsilon\end{equation} if $\delta<\epsilon$ and $m$ is sufficiently large.
		Since $m\in P_\eta^\lambda$, we know that
		$\frac{l_0}{2}-\eta\log_2m\leq\sum_{i=0}^{l_0-1}p(m)_i.$ Hence,
	
	\begin{equation}\label{maintheoremfifth} \frac{1}{3^{\sum_{j=0}^{l_0-1}p(m)_j}}3^{\frac{k}{2}}\leq\frac{1}{3^{\frac{l_0}{2}-\eta\log_2m}}3^{\frac{k}{2}}\leq m^{\eta\log_23}3^{\frac{k-l_0}{2}}.	\end{equation}
	As before, we can choose $\lambda,\eta>0$ small enough such that $m^{\eta\log_23}3^{\frac{k-l_0}{2}}\leq\frac{1}{2}m^\epsilon$ for sufficiently large $m$. Then, combining (\ref{maintheoremthird}), (\ref{maintheoremfourth}), and (\ref{maintheoremfifth}), we obtain $$r_k(m)3^{\frac{k}{2}}\leq m^\delta+m^{\eta\log_23}3^{\frac{k-l_0}{2}} <m^\epsilon.$$
	
		Hence, we are done showing that it suffices to show that the $S^\lambda_\epsilon$ are \stardense.
		By Lemma \ref{rest} and Lemma \ref{beginning}, we know that the set 	 \begin{align*}
			&S^{\log_2\sqrt{3}}_\epsilon=\\&\left\{m\in\positivenaturals\mid\forall 0\leq k\leq\log_2m:  \left(\cons\right)^km^{1-\epsilon}\leq \ColN^k(m)\leq \left(\cons\right)^km^{1+\epsilon}\wedge r_k(m)3^{\frac{k}{2}}m^{-\epsilon}<1\right\}\end{align*} is \stardense\ab for every $\epsilon>0$.
		Suppose that $\log_2\sqrt{3}<q<1$  and we know that $S^{\lambda}_\epsilon$ is \stardense\ab for some $\lambda\in (0,1]$ and all $\epsilon>0$. We will show that $S^{q\lambda}_\epsilon$ is \stardense\ab as well for every $\epsilon>0$.
		By Lemma \ref{Iterate}, we know that for any $\zeta>0$ the set $$S_\zeta^\prime=\left\{m\in\positivenaturals\mid \ColN^{\lfloor\log_2m\rfloor}(m)\in S_\zeta^{\lambda}\right\}$$ is \stardense.
		Suppose that $\delta,\zeta>0$ and $m\in S^{\log_2\sqrt{3}}_\delta\cap S_\zeta^\prime$. Then \begin{equation}\label{maintheoremfirst}
		\left(\cons\right)^k(\ColN^{\lfloor\log_2m\rfloor}(m))^{1-\zeta} \leq \ColN^k(\ColN^{\lfloor\log_2m\rfloor}(m))\leq \left(\cons\right)^k(\ColN^{\lfloor\log_2m\rfloor}(m))^{1+\zeta}\end{equation}
		for all	$0\leq k\leq \frac{1-\lambda}{1-\log_2\sqrt{3}}\log_2(\ColN^{\lfloor\log_2m\rfloor}(m))$.\\
		Furthermore, $$\left(\cons\right)^km^{1-\delta}\leq \ColN^k(m)\leq \left(\cons\right)^km^{1+\delta}$$ for all $0\leq k\leq\log_2m$. In particular,  \begin{equation}\label{maintheoremsecond}
	\left(\cons\right)^{\lfloor\log_2m\rfloor}m^{1-\delta}\leq \ColN^{\lfloor\log_2 m\rfloor}(m)\leq \left(\cons\right)^{\lfloor\log_2m\rfloor}m^{1+\delta}.\end{equation}
		Thus, if $0\leq k\leq\frac{1-\lambda}{1-\log_2\sqrt{3}}\log_2(\ColN^{\lfloor\log_2m\rfloor}(m))$, then we have on one side,
			\begin{align*}
		& \ColN^{k+\lfloor\log_2m\rfloor}(m)=\ColN^k(\ColN^{\lfloor\log_2m\rfloor}(m))\geq \left(\cons\right)^k(\ColN^{\lfloor\log_2m\rfloor}(m))^{1-\zeta}\\
		& \geq 	
		\left(\cons\right)^k\left(\left(\cons\right)^{\lfloor\log_2m\rfloor}m^{1-\delta}  \right)^{1-\zeta}=\left(\cons\right)^{k+\lfloor\log_2m\rfloor}\left( \left(\cons\right)^{\lfloor\log_2m\rfloor}\right)^{-\zeta} m^{(1-\delta)(1-\zeta)}.
		\end{align*}
		For the first inequality, we used (\ref{maintheoremfirst}) and (\ref{maintheoremsecond}) for the second . Similarly, on the other side, we get
			\begin{align*}
			& \ColN^{k+\lfloor\log_2m\rfloor}(m)=\ColN^k(\ColN^{\lfloor\log_2m\rfloor}(m))\leq \left(\cons\right)^k\left(\ColN^{\lfloor\log_2m\rfloor}(m)\right)^{1+\zeta}\\
			& \leq 	\left(\cons\right)^k\left(\left(\cons\right)^{\lfloor\log_2m\rfloor}m^{1+\delta}\right)^{1+\zeta}= \left(\cons\right)^{k+\lfloor\log_2m\rfloor}\left( \left(\cons\right)^{\lfloor\log_2m\rfloor}\right)^{\zeta} m^{(1+\delta)(1+\zeta)}.
		\end{align*}
Given any $\epsilon>0$, we can choose small enough $0<\delta<\epsilon$ and $\zeta>0$ such that for sufficiently large $m$ we have $$ \left(\left(\cons\right)^{\lfloor\log_2m\rfloor}\right)^{\zeta} m^{(1+\delta)(1+\zeta)}\leq m^{1+\epsilon},$$ and $$\left( \left(\cons\right)^{\lfloor\log_2m\rfloor}\right)^{-\zeta} m^{(1-\delta)(1-\zeta)}\geq m^{1-\epsilon}.$$
Thus, the set	$$W=\left\{m\in\positivenaturals\mid \forall 0\leq k-\lfloor\log_2m\rfloor\leq\frac{1-\lambda}{1-\log_2\sqrt{3}}\log_2(\ColN^{\lfloor\log_2m\rfloor}(m)): \left(\cons\right)^km^{1-\epsilon}\leq \ColN^k(m)\leq \left(\cons\right)^{k}m^{1+\epsilon}\right\}$$ is \stardense.
		Now, we have
		\begin{align*}
			&\lfloor\log_2 m\rfloor+\frac{1-\lambda}{1-\log_2\sqrt{3}}\log_2(\ColN^{\lfloor\log_2m\rfloor}(m))\\
			&\geq\log_2 m-1+\frac{1-\lambda}{1-\log_2\sqrt{3}}\log_2\left(\left(\cons\right)^{\lfloor\log_2m\rfloor}m^{1-\delta}\right)\\
			&	= \log_2m+\frac{1-\lambda}{1-\log_2\sqrt{3}}\log_2\left(\left(\cons\right)^{\lfloor\log_2m\rfloor}m\right)+\frac{1-\lambda}{1-\log_2\sqrt{3}}\log_2m^{-\delta}-1 \\
			&\geq \log_2m+\frac{1-\lambda}{1-\log_2\sqrt{3}}\left(\left(\log_2\sqrt{3}\right)\log_2m-\log_2m+\log_2m\right)+\frac{1-\lambda}{1-\log_2\sqrt{3}}\log_2m^{-\delta}-1 \\
			&= (1+\frac{1-\lambda}{1-\log_2\sqrt{3}}\log_2\sqrt{3})\log_2m+\frac{1-\lambda}{1-\log_2\sqrt{3}}\log_2m^{-\delta}-1\\
			&= \frac{1-\lambda\log_2\sqrt{3}}{1-\log_2\sqrt{3}}\log_2m+\frac{1-\lambda}{1-\log_2\sqrt{3}}\log_2m^{-\delta}-1.\\
		\end{align*}
	For the first inequality, we used that $\lfloor \log_2m\rfloor\geq \log_2m-1$ and (\ref{maintheoremsecond}), and for the second inequality, we used that $\left(\frac{\sqrt{3}}{2}\right)^{\lfloor\log_2m\rfloor}\geq \left(\frac{\sqrt{3}}{2}\right)^{\log_2m}$.
	Since $\log_2\sqrt{3}<q<1$, we conclude that $$\frac{1-\lambda\log_2\sqrt{3}}{1-\log_2\sqrt{3}}\log_2m+\frac{1-\lambda}{1-\log_2\sqrt{3}}\log_2m^{-\delta}-1\geq \frac{1-q\lambda}{1-\log_2\sqrt{3}}\log_2m$$ if $\delta>0$ is small enough and $m$ is sufficiently large. Thus, we have shown that the set
	$$V=\left\{m\in\positivenaturals\mid \lfloor\log_2 m\rfloor+\frac{1-\lambda}{1-\log_2\sqrt{3}}\log_2(\ColN^{\lfloor\log_2m\rfloor}(m))\geq \frac{1-q\lambda}{1-\log_2\sqrt{3}}\log_2m\right\} $$
	is \stardense. Thus, if $m$ is additionally in $W\cap V$, then 
	$$ \left(\cons\right)^km^{1-\epsilon}\leq \ColN^k(m)\leq \left(\cons\right)^km^{1+\epsilon}$$ for all $\lfloor\log_2m\rfloor\leq k\leq \frac{1-q\lambda}{1-\log_2\sqrt{3}}\log_2m$, and as $\delta<\epsilon$ and $m\in  S^{\log_2\sqrt{3}}_\delta$, these inequalities also hold for $0\leq k\leq \lfloor\log_2m\rfloor$.
	
	It remains to estimate $r_k(m)$.  By Lemma \ref{paritystar}, the set $$U_\delta=\left\{m\in\positivenaturals\mid \left(\frac{1}{2}-\delta\right)\lfloor\log_2m\rfloor\leq\sum_{k=0}^{\lfloor\log_2m\rfloor-1}p(m)_k\right\}$$ is \stardense\ab for any $\delta>0$. Assume that $m\in V\cap S^{\log_2\sqrt{3}}_\delta\cap S_\zeta^\prime\cap U_\delta$. Since $\delta<\epsilon$, we obtain
	$r_k(m)3^{\frac{k}{2}}m^{-\epsilon}<r_k(m)3^{\frac{k}{2}}m^{-\delta}<1$ for all $0\leq k\leq \log_2m$. Set $k_0=\lfloor\log_2 m\rfloor$. Since $m\in S_\zeta^\prime\cap S^{\log_2\sqrt{3}}_\delta\cap V$, we have 
	\begin{equation}\label{rkmestimateiteration}
	 r_k(\ColN^{k_0}(m))3^{\frac{k}{2}}<(\ColN^{k_0}(m))^{\zeta}\leq \left(\left(\cons\right)^{k_0}m^{1+\delta}\right)^\zeta	\end{equation}
	for all $0\leq k\leq \frac{1-q\lambda}{1-\log_2\sqrt{3}}\log_2m-k_0$.
	Furthermore, with 1. of Lemma \ref{splitrkm} for $k_0\leq k+k_0\leq \frac{1-q\lambda}{1-\log_2\sqrt{3}}\log_2m$, we get
		\begin{equation}\label{maintheoremeigth}
	r_{k+k_0}(m)
		=2^{-k}r_{k_0}(m)+\frac{1}{3^{\sum_{j=0}^{k_0-1}p(m)_j}}r_{k}(\ColN^{k_0}(m)).\end{equation}
		Now, since $m\in  S^{\log_2\sqrt{3}}_\delta$, we have
		\begin{equation}\label{maintheoremsixth}	2^{-k}r_{k_0}(m)3^{\frac{k+k_0}{2}}=\frac{3^{\frac{k}{2}}}{2^{k}}r_{k_0}(m)3^{\frac{k_0}{2}}\leq m^\delta.\end{equation} Furthermore, since $m\in U_\delta$ and using (\ref{rkmestimateiteration}), we get
		\begin{equation}\label{maintheoremseventh}\begin{split}
			&\frac{1}{3^{\sum_{j=0}^{k_0-1}p(m)_j}}r_{k}(\ColN^{k_0}(m))3^{\frac{k+k_0}{2}}\leq \frac{1}{3^{\frac{k_0}{2}-\delta\log_2m-\delta}}3^{\frac{k_0}{2}}r_{k}(\ColN^{k_0}(m))3^{\frac{k}{2}}\\&\leq \frac{3^\delta}{3^{-\delta\log_2m}}\left(\left(\cons\right)^{\lfloor\log_2m\rfloor}m^{1+\delta}\right)^\zeta\leq 3^{\delta}m^{(1+\delta)\zeta+\delta\log_23}.\end{split}\end{equation}
		For the last inequality, we used $\frac{\sqrt{3}}{2}<1$.
		If we choose $\zeta,\delta>0$ small enough and combining (\ref{maintheoremeigth}), (\ref{maintheoremsixth}, and (\ref{maintheoremseventh}) we can ensure that $$r_{k+k_0}(m)3^{\frac{k+k_0}{2}}\leq m^\delta+3^\delta m^{(1+\delta)\zeta+\delta\log_23}< m^\epsilon
		$$ if $m$ is large enough.

	Thus, we have shown that $S^{q\lambda}_{\epsilon}$ is \stardense. It follows inductively that $S^{q^n}_{\epsilon}$ is \stardense\ab for every $n\in\N$. Thus, given any $0<\lambda\leq 1$, we can find $n\in\N$ such that $q^n<\lambda$, thus, $S^{\lambda}_\epsilon$ is \stardense\ab as well since $S^{q^n}_\epsilon\subseteq S^{\lambda}_\epsilon$. Hence, the proof is complete.
	\end{proof}
\section{Bounds for the Stopping Time and Further Applications}
	In the following, we will mention a few consequences of Theorem \ref{secondmaintheorem}.
	First, there is the following reformulation of Theorem \ref{secondmaintheorem}, which is the version of Theorem \ref{reformnatural} for \stardensity\ab instead of natural density.
		\begin{thm}\label{reform}
		Suppose that $\epsilon>0$. Then the set $$\left\{m\in\positivenaturals\mid\forall \lambda\in[0,1]:  m^{\lambda-\epsilon}\leq \ColN^{\left\lfloor\frac{1-\lambda}{1-\log_2\sqrt{3}}\log_2m\right\rfloor}(m)\leq m^{\lambda+\epsilon}\right\}$$ is \stardense.
	\end{thm}
	\begin{proof}
		Suppose that $\epsilon,\delta>0$ and set $$A_\delta=\left\{m\in\positivenaturals\mid\forall 0\leq k\leq \frac{\log_2m}{1-\log_2\sqrt{3}}:  \left(\cons\right)^km^{1-\delta}\leq \ColN^k(m)\leq \left(\cons\right)^km^{1+\delta}\right\}.$$ By Theorem \ref{secondmaintheorem}, $A_\delta$ is \stardense. Suppose that $m\in A_\delta$ and $\lambda\in[0,1]$. Then $$0\leq \left\lfloor\frac{1-\lambda}{1-\log_2\sqrt{3}}\log_2m\right\rfloor\leq \frac{\log_2m}{1-\log_2\sqrt{3}},$$ thus, $$\left(\cons\right)^{\left\lfloor\frac{1-\lambda}{1-\log_2\sqrt{3}}\log_2m\right\rfloor}m^{1-\delta}\leq \ColN^{\left\lfloor\frac{1-\lambda}{1-\log_2\sqrt{3}}\log_2m\right\rfloor}(m)\leq \left(\cons\right)^{\left\lfloor\frac{1-\lambda}{1-\log_2\sqrt{3}}\log_2m\right\rfloor}m^{1+\delta}.$$
		We have on one side, $$\left(\cons\right)^{\left\lfloor\frac{1-\lambda}{1-\log_2\sqrt{3}}\log_2m\right\rfloor}m^{1+\delta}\leq\frac{2}{\sqrt{3}} \left(\cons\right)^{\frac{1-\lambda}{1-\log_2\sqrt{3}}\log_2m}m^{1+\delta}=\frac{2}{\sqrt{3}}m^{\lambda-1}m^{1+\delta}=\frac{2}{\sqrt{3}}m^{\lambda+\delta},$$
		and on the other side, $$\left(\cons\right)^{\left\lfloor\frac{1-\lambda}{1-\log_2\sqrt{3}}\log_2m\right\rfloor}m^{1-\delta}\geq \left(\cons\right)^{\frac{1-\lambda}{1-\log_2\sqrt{3}}\log_2m}m^{1-\delta}=m^{\lambda-1}m^{1-\delta}=m^{\lambda-\delta}.$$
		If $m$ is large enough and $\delta<\epsilon$, then we conclude that $$m^{\lambda-\epsilon}\leq m^{\lambda-\delta}\leq \ColN^{\left\lfloor\frac{1-\lambda}{1-\log_2\sqrt{3}}\log_2m\right\rfloor}(m)\leq \frac{2}{\sqrt{3}}m^{\lambda+\delta}\leq m^{\lambda+\epsilon},$$ therefore, the claim follows since $A_\delta$ is \stardense.
	\end{proof}
From Theorem \ref{reform} we get the slightly stronger \stardensity\ab version of Corollary \ref{upperbound}.

\begin{cor}
	Suppose that $\epsilon>0$. Then the set $$\left\{m\in\positivenaturals\mid \ColN^{\left\lfloor\frac{\log_2m}{1-\log_2\sqrt{3}}\right\rfloor}(m)\leq m^{\epsilon}\right\}$$  is \stardense.
\end{cor} 
As an immediate consequence of Theorem \ref{secondmaintheorem} and Lemma \ref{longtermparity}, we obtain the \stardensity\ab version of Theorem \ref{oddandeven}.
\begin{thm}\label{paritylongterm}
	Let $\epsilon>0$. Then the set  $$\left\{m\in\positivenaturals \mid\forall 0\leq k\leq \frac{\log_2m}{1-\log_2\sqrt{3}}:  -\epsilon\log_2m\leq\sum_{i=0}^{k-1}p(m)_i-\frac{k}{2}\leq \epsilon\log_2m\right\}$$ is \stardense.
\end{thm}
	Recall that for $m\in\positivenaturals$, we defined the total stopping time $\tau(m)$ to be the minimal $n\in\N$ such that $\ColN^n(m)=1$ if such an $n$ exists and $=\infty$ otherwise. First, we note the following:
	\begin{thm}\label{lowerboundstar}
		For any $0\leq\alpha<\frac{1}{1-\log_2\sqrt{3}}$ the set $\left\{m\in\positivenaturals \mid \tau(m)>\alpha\log_2(m)\right\}$ is \stardense.
	\end{thm}
	\begin{proof}
		Since $0\leq\alpha<\frac{1}{1-\log_2\sqrt{3}}$, there is a $\mu\in(0,1]$ such that $\alpha=\frac{1-\mu}{1-\log_2\sqrt{3}}$. Let $\epsilon>0$. By Theorem \ref{reform}, we know that  $$\left\{m\in\positivenaturals\mid\forall \lambda\in[0,1]:  m^{\lambda-\epsilon}\leq \ColN^{\left\lfloor\frac{1-\lambda}{1-\log_2\sqrt{3}}\log_2m\right\rfloor}(m)\leq m^{\lambda+\epsilon}\right\}$$ is \stardense, in particular, $$S=\left\{m\in\positivenaturals\mid m^{\mu-\epsilon}\leq \ColN^{\left\lfloor\frac{1-\mu}{1-\log_2\sqrt{3}}\log_2m\right\rfloor}(m)\right\}$$ is \stardense. If $0<\epsilon<\mu$, then $m^{\mu-\epsilon}>1$, hence, $\tau(m)> \left\lfloor\frac{1-\mu}{1-\log_2\sqrt{3}}\log_2m\right\rfloor=\lfloor\alpha\log_2m\rfloor$ for every $m\in S$.
	\end{proof}
Now, we can prove \ref{averagelowerbound}, ie., that there exists a lower bound for the average total stopping time, coinciding with the one conjectured by Crandall and Shanks (see \cite[page 13]{Lagariasgeneralizations}), i.e., 	$$\liminf_{x\rightarrow\infty}\frac{1}{x\log_2x}\sum_{k=1}^{\lfloor x\rfloor}\tau(k)\geq \frac{1}{1-\log_2\sqrt{3}}.$$
\begin{proof}[Proof of Theorem \ref{averagelowerbound}]
	It is enough to show that $$\liminf_{m\rightarrow\infty}\frac{1}{m\lfloor\log_2m\rfloor}\sum_{k=1}^{m}\tau(k)\geq \frac{1}{1-\log_2\sqrt{3}},$$ as $\lim_{x\rightarrow\infty}\frac{x\log_2x}{\lfloor x\rfloor\lfloor\log_2\lfloor x\rfloor\rfloor}=1$.
	Let $\delta>0$. By Theorem \ref{lowerboundstar} and the fact that $\lfloor\log_2m\rfloor\leq\log_2m$, we know that the set
	$$ \left\{m\in\positivenaturals\mid \tau(m)\geq \frac{1-\delta}{1-\log_2\sqrt{3}}\lfloor\log_2 m\rfloor\right\}$$
	is\stardense. By (the converse direction in) Lemma \ref{partitiondenseN}, we find $C>0$ and $0<D<1$ such that $$\mu_{\intervalN{2^n}{2^{n+1}}}\left(\left\{m\in\intervalN{2^n}{2^{n+1}}\mid \tau(m)\geq \frac{(1-\delta)n}{1-\log_2\sqrt{3}}\right\}\right)\geq 1-\frac{C}{2^{nD}}.$$ Suppose that $m\in\positivenaturals$. Put $n=\left\lfloor\log_2m\right\rfloor$. Then \begin{equation}\label{estimationtauk}
	\begin{split}
		&\sum_{k=1}^m\tau(k)=\sum_{i=0}^{n-1}\sum_{k=2^i}^{2^{i+1}-1}\tau(k)+\sum_{k=2^n}^m\tau(k)\\&\geq \sum_{i=0}^{n-1}\frac{(1-\delta)i}{1-\log_2\sqrt{3}}2^i\left(1-\frac{C}{2^{iD}}\right)+\frac{(1-\delta)n}{1-\log_2\sqrt{3}}\left(m-2^n-2^n\frac{C}{2^{nD}}\right).\end{split}\end{equation}
	Note that for any $x\in\R\setminus\left\{1\right\}$ and $n\in\N$ we have \begin{equation}\label{x}
		\sum_{k=1}^nkx^{k}=x\frac{(x-1)(n+1)x^n+1-x^{n+1}}{(1-x)^2},	\end{equation} in particular, for $x=2$ $$\sum_{k=1}^{n}k2^k=(n-1)2^{n+1}+2$$ for all $n\in\N$. Thus, \begin{align*}
		&\sum_{i=0}^{n-1}\frac{(1-\delta)i}{1-\log_2\sqrt{3}}2^i+\frac{(1-\delta)n}{1-\log_2\sqrt{3}}(m-2^n)=\frac{1-\delta}{1-\log_2\sqrt{3}}((n-2)2^{n}+2+n(m-2^n))\\&=\frac{1-\delta}{1-\log_2\sqrt{3}}(-2\cdot 2^{n}+2+nm).\end{align*} Thus, we have \begin{equation}\label{firstpart}
		\frac{1}{nm}\left(\sum_{i=0}^{n-1}\frac{(1-\delta)i}{1-\log_2\sqrt{3}}2^i+\frac{(1-\delta)n}{1-\log_2\sqrt{3}}(m-2^n)\right)=\frac{1}{nm}\frac{1-\delta}{1-\log_2\sqrt{3}}(-2\cdot 2^{n}+2+nm).\end{equation}
		Furthermore, equation (\ref{x}) with $x=2^{1-D}$ gives
	$$\sum_{i=1}^{n}i2^i\frac{1}{2^{iD}}=\sum_{i=1}^{n}i2^{(1-D)i}=2^{1-D}\frac{(2^{1-D}-1)(n+1)2^{(1-D)n}+1-2^{(1-D)(n+1)}}{(1-2^{1-D})^2}.$$ Thus, 	\begin{equation}\label{secondpart}
		\frac{1}{n2^n}\sum_{i=1}^{n}i2^i\frac{1}{2^{iD}}=2^{1-D}\frac{(2^{1-D}-1)(1+\frac{1}{n})2^{-Dn}+\frac{1}{n2^n}-\frac{1}{n}2^{1-D}2^{-Dn}}{(1-2^{1-D})^2}.	\end{equation}
	Since $D>0$, we have \begin{equation}\label{thirdpart}\lim_{n\rightarrow\infty}2^{(1-D)}\frac{(2^{(1-D)}-1)(1+\frac{1}{n})2^{-Dn}+\frac{1}{n2^n}-\frac{1}{n}2^{1-D}2^{-Dn}}{(1-2^{(1-D)})^2}=0.	\end{equation} Combining equations (\ref{firstpart}) and (\ref{secondpart}), estimate (\ref{estimationtauk}), (\ref{thirdpart}), and $2^n\leq m$, we obtain for arbitrary $\epsilon>0$ that
	$$\frac{1}{mn}\sum_{k=1}^m\tau(k)\geq \frac{1-\delta}{1-\log_2\sqrt{3}}-\epsilon$$ in case $m$ is large enough. Thus, 
	$$\liminf_{m\rightarrow\infty}\frac{1}{m\lfloor\log_2m\rfloor}\sum_{k=1}^{m}\tau(k)\geq \frac{1-\delta}{1-\log_2\sqrt{3}}.$$
Since $\delta>0$ was arbitrary, the claim follows.
\end{proof}
Next, we show that if the set of $m$ for which $\frac{\tau(m)}{\log_2m}$ is bounded by some $C>0$ is \stardense, then this is true for any bound $C>\frac{1}{1-\log_2\sqrt{3}}$.
	\begin{cor}\label{possibleupperbound}
		Suppose that there exists $C>0$ such that the set $$\left\{m\in\positivenaturals\mid \tau(m)\leq C\log_2m\right\}$$ is \stardense. Then for every $C^\prime>\frac{1}{1-\log_2\sqrt{3}}$ the set $$\left\{m\in\positivenaturals\mid \tau(m)\leq C^\prime\log_2m\right\}$$ is also \stardense.
	\end{cor}
\begin{proof}
	By Lemma \ref{Iterate} and Lemma \ref{beginning}, we know that the set  $$S=\left\{m\in\positivenaturals\mid \tau\left(\ColN^{\lfloor\log_2m\rfloor}(m)\right)\leq C\log_2\left(\ColN^{\lfloor\log_2m\rfloor}(m)\right) \wedge \ColN^{\lfloor\log_2m\rfloor}(m)\leq\left(\cons\right)^{\log_2m}m^{1+\epsilon}\right\}$$ is \stardense\ab for every $\epsilon>0$.
	If $m\in S$, then as $\tau(m)\geq \lfloor\log_2m\rfloor$ for all $m\in\positivenaturals$, we get \begin{equation}\label{taufirst}
		\tau(m)=\lfloor\log_2m\rfloor+\tau\left(\ColN^{\lfloor\log_2m\rfloor}(m)\right)\leq \lfloor\log_2m\rfloor+C\log_2\left(\ColN^{\lfloor\log_2m\rfloor}(m)\right)	\end{equation}
	and 
	\begin{equation}\label{tausecond}
		\log_2\left(\ColN^{\lfloor\log_2m\rfloor}(m)\right)\leq \log_2\left(\left(\cons\right)^{\log_2m}m^{1+\epsilon}\right)= (\log_2\sqrt{3})\log_2m+\epsilon\log_2m.	\end{equation}
	 Combining (\ref{taufirst}) and (\ref{tausecond}), we obtain that 
	 $$\left\{m\in\positivenaturals\mid \tau(m)\leq(1+ C\log_2\sqrt{3})\log_2m+C\epsilon\log_2m\right\}$$ is \stardense\ab for every $\epsilon>0$. Hence, we have shown that if for some $C>0$ 
	  $$\left\{m\in\positivenaturals\mid \tau(m)\leq C\log_2m\right\}$$ 
	 is \stardense, then also
	  $$\left\{m\in\positivenaturals\mid \tau(m)\leq D\log_2m\right\}$$ 
	 is \stardense\ab for  every $$D>1+C\log_2\sqrt{3}.$$ Iterating this, we get that for every $n\in\N$ and for every \begin{align*}
	 	D>C\left(\log_2\sqrt{3}\right)^{n+1}+\sum_{k=0}^n\left(\log_2\sqrt{3}\right)^k=C\left(\log_2\sqrt{3}\right)^{n+1}+\frac{1-\left(\log_2\sqrt{3}\right)^{n+1}}{1-\log_2\sqrt{3}}\end{align*}
	 the set 
	 $$\left\{m\in\positivenaturals\mid \tau(m)\leq D\log_2m\right\}$$ 
	 is \stardense. Hence, the claim follows as $\left(\log_2\sqrt{3}\right)^{n+1}\rightarrow 0$ for $n\rightarrow\infty$.
\end{proof}
In order to conclude that $$\lim_{x\rightarrow\infty}\frac{1}{x\log_2x}\sum_{k=1}^{\lfloor x\rfloor}\tau(k)= \frac{1}{1-\log_2\sqrt{3}},$$ a slightly weaker assumption than $\tau(m)$ being in $O(\log_2m)$ is sufficient. It is enough to assume that $\tau(m)\leq C\log_2m$ for a \stardense\ab set for some $C>0$ and to have $\tau(m)$ in $O(m^\alpha)$ for all $\alpha>0$. For example, this last condition is fulfilled if $\tau(m)$ is in $O((\log_2m)^K)$ for some $K\in\positivenaturals$.
\begin{cor}
		Suppose that there exists $C>0$ such that the set $$\left\{m\in\positivenaturals\mid \tau(m)\leq C\log_2m\right\}$$ is \stardense. Furthermore, suppose that there exist $K_\alpha\in(0,\infty)$ for all $\alpha>0$ such that $\tau(m)\leq K_\alpha m^\alpha$ for all $m\in\positivenaturals$. Then $$\lim_{x\rightarrow\infty}\frac{1}{x\log_2x}\sum_{k=1}^{\lfloor x\rfloor}\tau(k)= \frac{1}{1-\log_2\sqrt{3}}.$$
\end{cor}
\begin{proof}
	We already know that $$\liminf_{m\rightarrow\infty}\frac{1}{m\lfloor\log_2m\rfloor}\sum_{k=1}^{m}\tau(k)\geq \frac{1}{1-\log_2\sqrt{3}}$$ by Theorem \ref{averagelowerbound}. Thus, it is enough to show that $$\limsup_{m\rightarrow\infty}\frac{1}{m\lfloor\log_2m\rfloor}\sum_{k=1}^{m}\tau(k)\leq \frac{1}{1-\log_2\sqrt{3}}.$$ Towards this end, fix $\delta>0$. We are going to proceed as in the proof of Theorem \ref{averagelowerbound}. By Corollary \ref{possibleupperbound} and the fact that $$\lim_{m\rightarrow\infty}\frac{[\log_2m]}{\log_2m}=1,$$ we know that the set 
	 $$\left\{m\in\positivenaturals\mid \tau(m)\leq \frac{1+\delta}{1-\log_2\sqrt{3}}\lfloor\log_2m\rfloor\right\}$$ is \stardense. By Lemma \ref{partitiondenseN}, we find $C>0$ and $0<D<1$ such that \begin{equation}\label{estimatelimsup}
	 \mu_{\intervalN{2^n}{2^{n+1}}}\left(\left\{m\in\intervalN{2^n}{2^{n+1}}\mid \tau(m)\leq \frac{(1+\delta)n}{1-\log_2\sqrt{3}} \right\}\right)\geq 1-\frac{C}{2^{nD}}.\end{equation} Choose $0<\alpha<D$ and $K_\alpha>0$ such that $\tau(m)\leq K_\alpha 2^{\alpha\lfloor\log_2m\rfloor}$ for all $m\in\positivenaturals$. Fix $m\in\positivenaturals$ and put $n=\lfloor \log_2 m\rfloor$. Using (\ref{estimatelimsup}) and $\tau(k)\leq K_\alpha 2^{\alpha\lfloor\log_2k\rfloor}$ we get \begin{equation}\label{secontauestimation}
	 	\begin{split}
	 &\sum_{k=1}^m\tau(k)=\sum_{i=0}^{n-1}\sum_{k=2^i}^{2^{i+1}-1}\tau(k)+\sum_{k=2^n}^m\tau(k)\\&\leq \sum_{i=0}^{n-1}\frac{(1+\delta)i}{1-\log_2\sqrt{3}}2^i+\frac{(1+\delta)n}{1-\log_2\sqrt{3}}(m-2^n)+\sum_{i=0}^{n}K_\alpha 2^{i\alpha}C\cdot 2^{i(1-D)}.\end{split}\end{equation} As in the proof of Theorem \ref{averagelowerbound}, we obtain \begin{equation}\label{secondtaufirst}
\sum_{i=0}^{n-1}\frac{(1+\delta)i}{1-\log_2\sqrt{3}}2^i+\frac{(1+\delta)n}{1-\log_2\sqrt{3}}(m-2^n)=\frac{1+\delta}{1-\log_2\sqrt{3}}(-2\cdot 2^{n}+2+nm).\end{equation} It remains to estimate $\sum_{i=0}^{n}K_\alpha 2^{i\alpha}C\cdot 2^{i(1-D)}$. We have \begin{equation}\label{secondtausecond}
\sum_{i=0}^{n}K_\alpha 2^{i\alpha}C\cdot 2^{i(1-D)}=CK_\alpha\sum_{i=0}^{n} 2^{i(1-D+\alpha)}=CK_\alpha\frac{2^{(n+1)(1-D+\alpha)}-1}{2^{(1-D+\alpha)}-1}.\end{equation} Thus,
	\begin{equation}\label{resttauzero}
	\frac{1}{2^n}\sum_{i=0}^{n}K_\alpha 2^{i\alpha}C\cdot 2^{i(1-D)}=CK_\alpha\frac{2^{1-D+\alpha}2^{n(\alpha-D)}-2^{-n}}{2^{(1-D+\alpha)}-1}\rightarrow 0\end{equation} for $n\rightarrow\infty$ as $\alpha<D$. Thus, combining (\ref{secontauestimation}), (\ref{secondtaufirst}), (\ref{secondtausecond}), and (\ref{resttauzero}) we have for every $\delta>0$ $$\limsup_{m\rightarrow\infty}\frac{1}{m\lfloor\log_2 m\rfloor}\sum_{k=1}^m\tau(k)\leq \frac{1+\delta}{1-\log_2\sqrt{3}}$$ and the claim follows. 
\end{proof}
We now prove the \stardensity\ab version of Theorem \ref{thirdmaintheoremCol} using Theorem \ref{secondmaintheorem} and Theorem \ref{paritylongterm}.
\begin{thm}\label{thirdmaintheoremreformulation}
	Suppose that $\epsilon>0$. Then the set \begin{equation*}
		\left\{m\in\positivenaturals\mid\forall 0\leq k\leq \frac{3\log_2m}{2-\log_23}:  \left(\sqrt[3]{\frac{3}{4}}\right)^{k}m^{1-\epsilon}\leq \Col^k(m)\leq \left(\sqrt[3]{\frac{3}{4}}\right)^km^{1+\epsilon}\right\}\end{equation*} is \stardense.
\end{thm}
\begin{proof}
Suppose that $\epsilon>0$. Choose $\eta,\delta>0$ such that $\eta+\delta<\epsilon$.	By Theorem \ref{secondmaintheorem}, we know that the set
	\begin{equation*}
		Q=\left\{m\in\positivenaturals\mid\forall 0\leq k\leq \frac{\log_2m}{1-\log_2\sqrt{3}}:  \left(\cons\right)^km^{1-\eta}\leq \ColN^k(m)\leq \left(\cons\right)^km^{1+\eta}\right\}\end{equation*} is \stardense.
By Theorem \ref{paritylongterm}, we know that the set	$$S=\left\{m\in\positivenaturals \mid\forall 0\leq k\leq \frac{\log_2m}{1-\log_2\sqrt{3}}:  -\delta\log_2m\leq\sum_{i=0}^{k-1}p(m)_i-\frac{k}{2}\leq \delta\log_2m\right\}$$ is \stardense.
 Suppose that $m\in Q\cap S$. Note that $$\ColN^k(m)=\Col^{k+\sum_{i=0}^{k-1}p(m)_i}(m),$$ which follows easily by induction. Suppose that $ 0\leq k\leq \frac{3\log_2m}{2-\log_23}$, then  $0\leq \left\lfloor\frac{2k}{3}\right\rfloor\leq \frac{\log_2m}{1-\log_2\sqrt{3}}$. Since $m\in Q$, we get 
	\begin{equation}\label{thirdmainfirst}
		\left(\cons\right)^{\left\lfloor\frac{2k}{3}\right\rfloor}m^{1-\eta}\leq \ColN^{\left\lfloor\frac{2k}{3}\right\rfloor}(m)\leq \left(\cons\right)^{\left\lfloor\frac{2k}{3}\right\rfloor}m^{1+\eta}.
	\end{equation}
Now,
\begin{equation}\label{thirdmainsecond}
		\left(\sqrt[3]{\frac{3}{4}}\right)^km^{1-\eta}\leq\left(\cons\right)^{\left\lfloor\frac{2k}{3}\right\rfloor}m^{1-\eta},\end{equation} and
	\begin{equation}\label{thirdmain2}
		\left(\cons\right)^{\left\lfloor\frac{2k}{3}\right\rfloor}m^{1+\eta}\leq \frac{2}{\sqrt{3}}\left(\sqrt[3]{\frac{3}{4}}\right)^km^{1+\eta}.
\end{equation}
Thus, combing (\ref{thirdmainfirst}), (\ref{thirdmainsecond}), and (\ref{thirdmain2}) we obtain
\begin{equation}\label{thirdmainthird}
	\left(\sqrt[3]{\frac{3}{4}}\right)^km^{1-\eta}\leq \ColN^{\left\lfloor\frac{2k}{3}\right\rfloor}(m)\leq \frac{2}{\sqrt{3}}\left(\sqrt[3]{\frac{3}{4}}\right)^km^{1+\eta}.
\end{equation}
Similar as in Lemma \ref{bruteforce}, it is easy to see that 
\begin{equation}\label{thirdmainfourth}
\frac{m}{2^k}\leq \Col^k(m)\leq 2\cdot 2^km.
\end{equation}
Since $m\in S$, we also know that
\begin{equation}\label{thirdmainfifth}
	-\delta\log_2m\leq\sum_{i=0}^{\left\lfloor\frac{2k}{3}\right\rfloor-1}p(m)_i-\frac{\left\lfloor\frac{2k}{3}\right\rfloor}{2}\leq \delta\log_2m.
\end{equation} 
Thus,
\begin{equation}\label{thirdmainsixth}
\ColN^{\left\lfloor\frac{2k}{3}\right\rfloor}(m)=\Col^{\left\lfloor\frac{2k}{3}\right\rfloor+\sum_{i=0}^{\left\lfloor\frac{2k}{3}\right\rfloor-1}p(m)_i}(m)=\Col^{\frac{3}{2}\left\lfloor\frac{2k}{3}\right\rfloor+\left(\sum_{i=0}^{\left\lfloor\frac{2k}{3}\right\rfloor-1}p(m)_i-\frac{1}{2}\left\lfloor\frac{2k}{3}\right\rfloor\right)}(m)=\Col^{k+j}(m)
\end{equation}
for some  $-\delta\log_2m-1\leq j\leq \delta\log_2m$ by (\ref{thirdmainfifth}) and $k-1\leq\frac{3}{2}\left\lfloor\frac{2k}{3}\right\rfloor\leq k$. Considering  (\ref{thirdmainfourth}), (\ref{thirdmainsixth}), and (\ref{thirdmainthird}), we obtain
\begin{equation*}\label{thirdmaineventh}
\Col^{k}(m)\leq 2\cdot2^{|j|}\Col^{k+j}(m)\leq 2\cdot 2^{\delta\log_2m+1}\ColN^{\left\lfloor\frac{2k}{3}\right\rfloor}(m)\leq 4m^{\delta}\frac{2}{\sqrt{3}}\left(\sqrt[3]{\frac{3}{4}}\right)^km^{1+\eta}
\end{equation*}
and  
\begin{equation*}\label{thirdmaineight}
	\Col^{k}(m)\geq 2^{-|j|-1}\Col^{k+j}(m)\geq 2^{-\delta\log_2m-2}\ColN^{\left\lfloor\frac{2k}{3}\right\rfloor}(m)\geq m^{-\delta}2^{-2}\left(\sqrt[3]{\frac{3}{4}}\right)^km^{1-\eta}.
\end{equation*}
Since $\delta+\eta<\epsilon$, we obtain for any sufficiently large $m$
$$m^{\delta}\frac{8}{\sqrt{3}}m^{1+\eta}\leq m^{1+\epsilon}$$
and 
$$ m^{-\delta}2^{-2}m^{1-\eta}\geq m^{1-\epsilon}.$$
Hence, the claim follows.\end{proof}
At last, we prove the \stardensity\ab version of Theorem \ref{TheoremColfast}.
 First, define a set $S\subseteq \oddnaturals$ to be \textbf{\stardense} in $\oddnaturals$ if there exist $0<D<1$ and $C>0$ such that $$\mu_{\intervalA{1}{2N+1}\cap \oddnaturals}(S\cap\intervalA{1}{2N+1}\cap \oddnaturals )\geq 1-\frac{C}{(N+1)^D}$$ for all $N\in\N$. It is clear that such a set is of natural density $1$ in $\oddnaturals$.
\begin{thm}
	Suppose that $\epsilon>0$. Then the set
	\begin{equation*}
	\left\{m\in \oddnaturals\mid\forall 0\leq k\leq\left(\log_2\frac{4}{3}\right)^{-1}\log_2m:  \left(\frac{3}{4}\right)^{k}m^{1-\epsilon}\leq \Colfast^k(m)\leq \left(\frac{3}{4}\right)^km^{1+\epsilon}\right\}\end{equation*} is \stardense\ab in $\oddnaturals$.
\end{thm}

\begin{proof}
	Let $\epsilon>0$. Choose $\delta,\eta>0$ such that $5\delta+\eta<\epsilon$.
	By Theorem \ref{secondmaintheorem}, we know that the set
	\begin{equation*}
		Q=\left\{m\in\positivenaturals\mid\forall 0\leq k\leq \frac{\log_2m}{1-\log_2\sqrt{3}}:  \left(\cons\right)^km^{1-\eta}\leq \ColN^k(m)\leq \left(\cons\right)^km^{1+\eta}\right\}\end{equation*} is \stardense.
	By Theorem \ref{paritylongterm}, we know that the set	$$S=\left\{m\in\positivenaturals \mid\forall 0\leq k\leq \frac{\log_2m}{1-\log_2\sqrt{3}}:  -\delta\log_2m\leq\sum_{i=0}^{k-1}p(m)_i-\frac{k}{2}\leq \delta\log_2m\right\}$$ is \stardense.
	Thus, also $Q\cap S$ is \stardense. Let $0<D<1$ and $C>0$ such that $Q\cap S$ has \density{C,D}. Put $R=\positivenaturals\setminus (Q\cap S)$. Then $\#R\cap\intervalA{1}{2N+1}\leq C (2N+1)^{1-D}$. Thus, also $$\#R\cap\intervalA{1}{2N+1}\cap\oddnaturals\leq C\cdot (2N+1)^{1-D},$$ and hence, $$\#Q\cap S\cap\intervalA{1}{2N+1}\cap\oddnaturals\geq N+1- C\cdot(2N+1)^{1-D}.$$ Therefore, $$\frac{\#Q\cap S\cap\intervalA{1}{2N+1}\cap\oddnaturals}{N+1}\geq 1- C\frac{(2N+1)^{1-D}}{N+1}\geq 1-C\cdot2^{1-D}(N+1)^{-D}.$$ We have shown that $Q\cap S\cap\oddnaturals$ is \stardense\ab in $\oddnaturals$.
	
	Note that, by definition of $\Colfast$, if  $m\in\oddnaturals,$ $j\in\N$, and  $p(m)_j=1$, then 
	\begin{equation}\label{fastacc1}
		\Colfast^{\sum_{i=0}^{j-1}p(m)_i}(m)=\ColN^j(m).\end{equation} As in Lemma \ref{bruteforce}, we get for $m\in\oddnaturals$ and $j\in\N$ \begin{equation}\label{fastacc2}
		\Colfast^j(m)\leq 2^jm.\end{equation}
	
	Assume that $m\in Q\cap S\cap\oddnaturals$. Define $k_0=\left\lfloor\left(\log_2\frac{4}{3}\right)^{-1}\log_2m\right\rfloor$, $L_0=\left\lfloor\frac{\log_2m}{1-\log_2\sqrt{3}}\right\rfloor$, and $k_{\max}=\sum_{i=0}^{L_0}p(m)_i-1$.\footnote{Note that $k_{\max}$ is close to $k_0$ as we will see below, in particular, $k_{\max}\geq 0$ in case $m$ is large enough and we choose $\delta$ small enough.} Let $0\leq L\leq L_0$ be maximal with $p(m)_L=1$. Then $\sum_{i=0}^{L-1}p(m)_i=k_{\max}$ and $\Colfast^{k_{\max}}(m)=\ColN^L(m)$ by (\ref{fastacc1}). Since $m\in S$, we obtain
	  \begin{equation}\label{fastacc3}
	  -\delta\log_2m\leq\sum_{i=0}^{L_0-1}p(m)_i-\frac{L_0}{2}=k_{\max}+1-p(m)_{L_0}-\frac{L_0}{2}\leq k_{\max}-\frac{L_0}{2}+1\end{equation}
	and  $$\sum_{i=0}^{L-1}p(m)_i-\frac{L}{2}=k_{\max}-\frac{L}{2}\leq \delta\log_2m.$$
	Thus,
	\begin{equation}\label{fastacc4}
0\leq L_0-L\leq 4\delta\log_2m+2.	\end{equation}
Furthermore, it follows from the definition of $k_0$ and $L_0$ that $k_0\leq \frac{L_0}{2}$. Using (\ref{fastacc3}) we conclude
\begin{equation}\label{fastacc5}
k_0- k_{\max}\leq \delta\log_2m	+1.
\end{equation}
Suppose that $k_{\max}< k_0$ and let $k_{\max}<k\leq k_0$. Then, using (\ref{fastacc2}), (\ref{fastacc5}), $\Colfast^{k_{\max}}(m)=\ColN^L(m)$, (\ref{fastacc4}), and $m\in Q$, we obtain \begin{equation}\label{fastaccseventh}
\Colfast^k(m)\leq 2^{k-k_{\max}}\Colfast^{k_{\max}}(m)\leq 2^{\delta\log_2m+1}\ColN^L(m)\leq 2m^{\delta}2^{4\log_2m+2}\ColN^{L_0}(m)\leq  m^{5\delta}\frac{16}{\sqrt{3}}m^{\eta}.\end{equation}
Furthermore, using $k_{\max}<k$, (\ref{fastacc5}), $\frac{4}{3}<2$, $\delta<\epsilon$, and $\Colfast^k(m)\geq 1$, we get 
\begin{equation}\label{fastacc8}
	\begin{split}
&\left(\frac{3}{4}\right)^{k}m^{1-\epsilon}\leq \left(\frac{3}{4}\right)^{k_{\max}}m^{1-\epsilon}\leq \left(\frac{3}{4}\right)^{\left(\log_2\frac{4}{3}\right)^{-1}\log_2m-\delta\log_2m-2}m^{1-\epsilon}=m^{-1}\left(\frac{4}{3}\right)^{\delta\log_2m+2}m^{1-\epsilon}\\
&\leq 2^{\delta\log_2m+2}m^{-\epsilon}=4m^{\delta-\epsilon}\leq 1\leq \Colfast^k(m)\end{split}\end{equation}
if $m$ is large enough.
Thus, since $5\delta+\eta<\epsilon$, combining (\ref{fastaccseventh}) and (\ref{fastacc8}), we have shown that $$\left(\frac{3}{4}\right)^{k}m^{1-\epsilon}\leq \Colfast^k(m)\leq \left(\frac{3}{4}\right)^km^{1+\epsilon}$$
in case $k_{\max}<k\leq k_0$.

Now, suppose that $0\leq k\leq \min\left(k_0,k_{\max}\right)$. Then there exists $J\leq L_0$ with $p(m)_J=1$ and $\sum_{i=0}^{J-1}p(m)_i=k$.
Thus, $\Colfast^{k}(m)=\ColN^J(m)$ by (\ref{fastacc1}).
Now, using once more that $m\in S$, we obtain
$$ -\delta\log_2m\leq k-\frac{J}{2}=\sum_{i=0}^{J-1}p(m)_i-\frac{J}{2}\leq \delta\log_2m.$$
Thus,
$$-2\delta\log_2m\leq J-2k\leq 2\delta\log_2m.$$
Hence, by Lemma \ref{bruteforce}, we get
\begin{equation}\label{fastversionx}
	m^{-2\delta}\ColN^{2k}(m)=2^{-2\delta\log_2m}\ColN^{2k}(m)\leq \ColN^J(m)\leq 	2^{2\delta\log_2m}\ColN^{2k}(m)=m^{2\delta}\ColN^{2k}(m).\end{equation}
Since $m\in Q$, we conclude
$$ m^{1-\eta}\left(\frac{3}{4}\right)^k\leq \ColN^{2k}(m)\leq  m^{1+\eta}\left(\frac{3}{4}\right)^k,$$ and with (\ref{fastversionx}) we get
\begin{equation*}\label{fastversion3}
	 m^{1-\eta-2\delta}\left(\frac{3}{4}\right)^k\leq \ColN^{J}(m)\leq  m^{1+\eta+2\delta}\left(\frac{3}{4}\right)^k.
\end{equation*}
Thus, since $\Colfast^{k}(m)=\ColN^{J}(m)$ and $\eta+2\delta<\epsilon$,
\begin{equation*}
	  m^{1-\epsilon}\left(\frac{3}{4}\right)^k\leq m^{1-\eta-2\delta}\left(\frac{3}{4}\right)^k\leq \Colfast^{k}(m)\leq m^{1+\eta+2\delta}\left(\frac{3}{4}\right)^k\leq m^{1+\epsilon}\left(\frac{3}{4}\right)^k,
\end{equation*}
which concludes the proof.
\end{proof}
\begin{rem}
	The methods of this paper are not restricted to the analysis of \ab$\ColN$ as there are maps that behave similar to it. As an example, take a tuple $\overline{a}=(p,q_0,...,q_{p-1},k_0,...,k_{p-1})$ such that $p\in\positivenaturals$, $q_i\in\positivenaturals$, $k_i\in\Z$, and $q_ii+k_i\equiv 0\mod p$. Define $$\ColN_{\overline{a}}:\Z\rightarrow\Z;\ab m\mapsto \frac{q_im+k_i}{p}\text{\hspace{10pt} if\hspace{10pt}} m\equiv i \mod p.$$ (To the author's knowledge  maps of this form were first introduced in \cite{Conway72}). Heuristically, $\ColN_{\overline{a}}$ is decreasing on average if $\frac{\sqrt[p]{\Pi_{i=0}^{p-1}q_i}}{p}<1$. In this case, it is possible to adjust the methods of this paper to prove an analogue of Theorem \ref{main}.
\end{rem}

\subsection*{Acknowledgments}
I thank Jeffrey Lagarias for pointing out relevant references and suggestions to improve the exposition of the introduction. Furthermore, I thank Tung Nguyen, Claudius Röhl, and Natalie Weißenböck for useful comments on earlier versions of this paper.
			\bibliography{Voll}
		\bibliographystyle{plain}
		\textit{email address:} Manuel.Inselmann@gmx.de
\end{document}